\documentclass[abstract=on]{scrartcl}
\setkomafont{sectioning}{\normalfont\normalcolor\bfseries}
\usepackage{amsmath}
\usepackage{amssymb}
\usepackage{amsthm}
\usepackage{mathtools}
\usepackage[T1]{fontenc}
\usepackage[utf8]{inputenc}
\usepackage{setspace}
\usepackage{graphicx}
\usepackage{color}
\usepackage{cite}
\usepackage{subfig}
\usepackage{bbm} 
\usepackage{enumitem} 
\usepackage{tikz}
\usepackage{mathrsfs}
\usepackage{algorithm}
\usepackage{algpseudocode}
\usepackage[hyperfootnotes=false,hidelinks,bookmarksnumbered,bookmarksopen]{hyperref}
\hypersetup{colorlinks=false}
\usepackage{nameref}
\usepackage{multirow}
\usepackage{booktabs}
\usepackage{mathdots}
\usepackage{varwidth}
\usepackage{siunitx} 
\usepackage{url}

\usetikzlibrary{patterns}
\usetikzlibrary{decorations.pathreplacing}

\setenumerate{label=(\roman*)}

\usepackage{geometry}
\newtheoremstyle{mythm}
   {10pt}                   
   {10pt}                   
   {}          		    
   {}                      
   {\normalfont\bfseries}  
   {}                      
   { }
   {\textbf{\thmname{#1} \thmnumber{#2} \thmnote{(#3)}}} 

\newtheoremstyle{mythm2}
   {10pt}                   
   {10pt}                   
   {\itshape}          		    
   {}                      
   {\normalfont\bfseries}  
   {}                      
   { }
   {\textbf{\thmname{#1} \thmnumber{#2} \thmnote{(#3)}}} 

\newtheoremstyle{myex}
  {10pt}                   
   {10pt}                   
   {}          		    
   {}                      
   {\normalfont\bfseries}  
   {}                      
   { }
   {\textbf{\thmname{#1} \thmnumber{#2} \thmnote{(#3)}}} 

\theoremstyle{myex}
\newtheorem*{XxmpX}{Example} 
\newenvironment{ex}    
  {%
   \pushQED{\qed}\begin{XxmpX}}
  {\popQED\end{XxmpX}}
   
\theoremstyle{mythm2}
\newtheorem{thm}{Theorem}[section]

\newtheorem{lem}[thm]{Lemma}
\theoremstyle{mythm}

\newtheorem{Rem}[thm]{Remark} 
\newenvironment{rem}    
  {%
   \pushQED{\qed}\begin{Rem}}
  {\popQED\end{Rem}}

\DeclareMathOperator{\argmax}{argmax}
\DeclareMathOperator{\diag}{diag}

\newcommand{\eqcolon}{\mathrel{\resizebox{\widthof{$\mathord{=}$}}{\height}{ $\!\!=\!\!\resizebox{1.2\width}{0.8\height}{\raisebox{0.23ex}{$\mathop{:}$}}\!\!$ }}}

\newcommand{\mods}{\,\bmod\,}
\newcommand{\RR}{\mathbb{R}}
\newcommand{\CC}{\mathbb{C}}
\newcommand{\NN}{\mathbb{N}}
\newcommand{\x}{\mathbf{x}}
\newcommand{\y}{\mathbf{y}}
\newcommand{\cxx}{\mathbf{x}^{\widehat{\mathrm{II}}}}
\newcommand{\cx}{x^{\widehat{\mathrm{II}}}}
\newcommand{\cfxx}{\mathbf{x}^{\widehat{\mathrm{IV}}}}
\newcommand{\dct}{^{\widehat{\mathrm{II}}}}
\newcommand{\dctf}{^{\widehat{\mathrm{IV}}}}
\newcommand{\jj}{^{(j)}}
\newcommand{\jjj}{^{(j+1)}}
\newcommand{\uu}{\mathbf{u}}
\newcommand{\z}{\mathbf{z}}
\newcommand{\V}{\mathbf{V}}
\newcommand{\PP}{\mathbf{P}}
\newcommand{\I}{\mathbf{I}}
\newcommand{\T}{\mathbf{T}}
\newcommand{\bfeta}{\boldsymbol{\eta}}
\newcommand{\eps}{\varepsilon}
\newcommand{\J}{\mathbf{J}}

\newcommand{\D}{\mathbf{D}}
\newcommand{\Ct}{\mathbf{C}^{\mathrm{II}}}
\newcommand{\Cf}{\mathbf{C}^{\mathrm{IV}}}
\newcommand{\Sf}{\mathbf{S}^{\mathrm{IV}}}
\newcommand{\cc}{\mathbf{c}}
\newcommand{\s}{\mathbf{s}}
\newcommand{\bb}{\mathbf{b}}
\newcommand{\A}{\mathbf{A}}

\begin{document}	
\title{Real Sparse Fast DCT for Vectors with Short Support}
\author{Sina Bittens\thanks{University of Göttingen, Institute for Numerical and Applied Mathematics, Lotzestr. 16-18, 37083 Göttingen, 
Germany. Email: sina.bittens@mathematik.uni-goettingen.de} \and Gerlind Plonka\thanks{University of Göttingen, Institute for Numerical and Applied Mathematics, Lotzestr. 16-18, 37083 Göttingen, 
Germany. Email: plonka@math.uni-goettingen.de} }
\date{\today}
\maketitle

\begin{center}
 \textsc{Dedicated to Manfred Tasche on the occasion of his 75th birthday}
\end{center}
\vspace{1cm}

\begin{abstract}
 In this paper we present a new fast and deterministic algorithm for the inverse discrete cosine transform of type II for reconstructing the input vector $\x\in\RR^N$, $N=2^J$, with short support of length $m$ from its discrete cosine transform $\cxx=\Ct_N\x$ if an upper bound $M\geq m$ is known. The resulting algorithm only uses real arithmetic,  has a runtime of $\mathcal{O}\left(M\log M+m\log_2\frac{N}{M}\right)$ and requires $\mathcal{O}\left(M+m\log_2\frac{N}{M}\right)$ samples of $\cxx$. For $m,M\rightarrow N$ the runtime and sampling requirements approach those of a regular IDCT-II for vectors with full support. The algorithm presented hereafter does not employ inverse FFT algorithms to recover $\x$.
\end{abstract}
\textbf{Keywords.}\quad discrete cosine transform, deterministic sparse fast DCT, sublinear sparse DCT
\newline
\textbf{AMS Subject Classification.}\quad 65T50, 42A38, 65Y20.
\section{Introduction}
Due to recent efforts deterministic sparse FFT algorithms utilizing a priori knowledge of the resulting vector are now well established,  and there exist several methods which achieve runtimes that scale sublinearly in the vector length $N$ if $\x\in\CC^N$ is known to possess at most $m$ significantly large entries. If, for example, the support of a vector $\x\in\RR^{2^J}_{\geq0}$ has a short support of length $m$, there exists a deterministic, adaptive DFT algorithm with runtime $\mathcal{O}\left(m\log m\log\frac{N}{m}\right)$, see \cite{plonka_nonneg}. Other deterministic, sublinear-time methods with different requirements on the sought-after vector $\x$ and its support structure include \cite{bit_zha_iw,akavia2014,bittens2017,christlieb2016multiscale,iwen,iwen_improved,plonka_smallsupp,plonka_nonneg,plonka_sparse,iwen2013improved,bit_plon}. 

The investigation of sparse and fast deterministic algorithms for the related trigonometric transforms in their respective cosine and sine bases has not yet been that thorough. However, besides the DFT, the discrete cosine transform (DCT) is one of the most important algorithms in engineering and data processing. Among numerous other applications the sparse DCT can be employed to evaluate polynomials in monomial form from sparse expansions of Chebyshev polynomials, see, e.g., \cite{PPST}, Chapter 6. As far as we are aware, there exist no fast sparse methods that have been specifically optimized for the cosine or sine bases. Of course it is always possible to apply sparse FFT algorithms to obtain for example $\x\in\RR^N$ from $\cxx$, using that 
\begin{equation}\label{eq:dft_dct}
 x_k\dct=\frac{\eps_N(k)}{\sqrt{2N}}\omega_{4N}^k\cdot \widehat{y}_k, \qquad \forall k\in\{0,\dotsc,N-1\},
\end{equation}
where $\eps_N(k)=\frac{1}{\sqrt{2}}$ for $k\equiv 0\mods N$ and $\eps_N(k)=1$ for $k\not\equiv 0\mods N$, $\omega_{4N}=e^{\frac{-2\pi i}{N}}$ and $\y=\left(x_0,x_1,\dotsc,x_{N-1},x_{N-1},x_{N-2},\dotsc,x_0\right)^T\in\RR^{2N}$, see, e.g., \cite{bit_plon} and \cite{PPST}, Chapter 6.4.1. However, if $\x$ is $m$-sparse, then $\y$ is $2m$-sparse, so applying a general sparse FFT algorithm is not the most efficient solution. Furthermore, $\y$ is symmetric and its support structure is closely related to the support structure of $\x$, which can be used to improve the runtime. In \cite{bit_plon}, where the recovery of a vector $\x\in\RR^N$ with short support of length $m$ from $\cxx$ based on (\ref{eq:dft_dct}) is studied, the short support of $\x$ and the resulting symmetric reflected block support of $\y$ are exploited. The algorithm in \cite{bit_plon} achieves a sublinear runtime of $\mathcal{O}\left(m\log m\log\frac{2N}{m}\right)$ and requires $\mathcal{O}\left(m\log\frac{2N}{m}\right)$ samples of the input vector $\cxx\in\RR^N$, $N=2^J$. Thus it performs better than general sparse FFT methods, as it is specifically tailored to the occurring support structure. Nevertheless, despite being an adaptive algorithm which does not need any a priori knowledge of the support length, its assumptions on the sought-after vector $\x$ are quite strict and, without supposing extensive knowledge of $\x$, they can usually only be satisfied if, e.g., $\x\in\RR^N_{\geq0}$. Furthermore, the algorithm relies on complex arithmetic, as the problem of reconstructing $\x$ from $\cxx$ is transferred to the problem of reconstructing the vector $\y\in\RR^{2N}$ of double length from its Fourier transform $\widehat{\y}\in\CC^{2N}$, which can be computed efficiently from $\cxx$. 

However, since there also exist fast DCT algorithms for arbitrary vectors that are completely based on real arithmetic, investigating fully real sparse fast DCT algorithms is the natural next course of action. In this paper we present the, to the best of our knowledge, first deterministic sparse fast algorithm for the inverse DCT-II (or, equivalently, for the DCT-III) that only employs real arithmetic. To be more precise we assume that the vector $\x\in\RR^N$, $N=2^J$, which we want to reconstruct, has a short support, or one-block support, of length $m<N$ and that an upper bound $M\geq m$ on the support length is known a priori. If the vector additionally satisfies the simple non-cancellation condition that the first and last entry in the support do not sum up to zero, the algorithm proposed herein recovers $\x$ exactly in $\mathcal{O}\left(M\log M+m\log_2\frac{N}{M}\right)$ time, for which $\mathcal{O}\left(M+m\log_2\frac{N}{M}\right)$ samples of $\cxx$ are required. Thus, if $m,M\rightarrow N$, the algorithm approaches the same runtime and sampling requirements as a regular IDCT-II for vectors of length $N$ with full support.

\subsection{Notation and Problem Statement}
Let $N=2^J$ with $J\in\NN$. For $a,b\in\NN_0$, $a\leq b$, we denote by $I_{a,b}$ the set
\[
 I_{a,b}\coloneqq\{a,a+1,\dotsc,b\}\subset\NN_0
\]
of integers. We say that a vector $\x=\x^{(J)}=\left(x_k\right)_{k=0}^{N-1}\in\RR^N$ has a \emph{short support}, or \emph{one-block support}, $S^{(J)}$ \emph{of length} $m^{(J)}=m$ if
\[
 x_k=0 \qquad \forall\, k\notin S^{(J)}\coloneqq I_{\mu^{(J)},\nu^{(J)}}=\left\{\mu^{(J)},\mu^{(J)}+1,\dotsc,\nu^{(J)}\right\},
\]
for some $\mu^{(J)}\in\{0,\dotsc,N-m\}$ and $\nu^{(J)}\coloneqq\mu^{(J)}+m-1$ with $x_{\mu^{(J)}}\neq 0$ and $x_{\nu^{(J)}}\neq0$. Note that, unlike in \cite{bit_plon}, we do not allow a periodic support in this paper. The interval $S^{(J)}\coloneqq I_{\mu^{(J)},\nu^{(J)}}$ is called the \emph{support interval}, $\mu^{(J)}$ the \emph{first support index} and $\nu^{(J)}$ the \emph{last support index} of $\x$. The support length and the first and last support index are uniquely determined. 

For $n\in\NN$ the \emph{cosine matrix of type II} is defined as
\begin{equation*}
  \Ct_n\coloneqq\sqrt{\frac{2}{n}}\left(\eps_n(k)\cos\left(\frac{k(2l+1)\pi}{2n}\right)\right)_{k,\,l=0}^{n-1},
\end{equation*}
 where $\eps_n(k)\coloneqq \frac{1}{\sqrt{2}}$ for $k\equiv 0\mods n $ and $\eps_n(k)\coloneqq 1$ for $k\not\equiv 0\mods n $.
This matrix is orthogonal, i.e., $ \Ct_n \left(\Ct_n\right)^{T} = {\mathbf I}_{n}$, where ${\mathbf I}_{n}$ denotes  the identity matrix of size $n \times n$. 
The \emph{discrete cosine transform of type II (DCT-II)} of $\x\in\RR^n$ is given by
 \[
  \cxx\coloneqq\Ct_n\x.
 \]
The inverse DCT-II coincides with the \emph{discrete cosine transform of type III (DCT-III)} with transformation matrix ${\mathbf C}_n^{\mathrm{III}} \coloneqq \left(\Ct_n\right)^{T}$. The \emph{cosine matrix of type IV} is defined as
 \[
  \Cf_n\coloneqq\sqrt{\frac{2}{n}}\left(\cos\left(\frac{(2k+1)(2l+1)\pi}{4n}\right)\right)_{k,\,l=0}^{n-1}.
 \]
 This matrix is orthogonal as well, with $\Cf_n=\left(\Cf_n\right)^T$, and the \emph{discrete cosine transform of type IV (DCT-IV)} of $\x\in\RR^n$ is given by
 \[
  \cfxx\coloneqq\Cf_n\x.
 \]
 Furthermore, the closely related \emph{sine matrix of type IV} is defined as
 \[
  \Sf_n\coloneqq\sqrt{\frac{2}{n}}\left(\sin\left(\frac{(2k+1)(2l+1)\pi}{4n}\right)\right)_{k,\,l=0}^{n-1}.
 \]
The purpose of this paper is to develop a deterministic sparse fast DCT algorithm for recovering $\x\in\RR^N$ with (unknown) short support of length $m<N$ from its DCT-II, $\cxx$, in sublinear time $\mathcal{O}\left(M\log M+m\log_2\frac{N}{M}\right)$ if an upper bound $M\geq m$ on the support length of $\x$ is known. If $m$ or $M$ approach the vector length $N$, the algorithm introduced herein still has a runtime complexity of $\mathcal{O}(N\log N)$ which is also achieved by fast DCT algorithms for vectors with full support, see, e.g., \cite{plonka_dct,wang}. For exact data our algorithm returns the correct vector $\x$ if, in addition to $x_{\mu^{(J)}}\neq 0$ and $x_{\nu^{(J)}}\neq0$, $\x$ satisfies the non-cancellation condition 
\begin{equation}\label{eq:suppose}
 x_{\mu^{(J)}}+x_{\nu^{(J)}}\neq0 \qquad \text{if }m\text{ is even.}
\end{equation}
This condition holds for example if all nonzero entries of $\x$ are positive or if all nonzero entries of $\x$ are negative, i.e., $\x\in\RR^N_{\geq0}$ or $\x\in\RR^N_{\leq0}$. In practice, i.e., for noisy data, one has to guarantee that for a threshold $\eps>0$ depending on the noise level we have
\[
 \left|x_{\mu^{(J)}}\right|>\eps, \; \left|x_{\nu^{(J)}}\right|>\eps \qquad\text{and}\qquad \left|x_{\mu^{(J)}}+x_{\nu^{(J)}}\right|>\eps.
\]
\subsection{Outline of the Paper}
The algorithm presented in this paper generalizes ideas introduced in \cite{plonka_smallsupp,plonka_nonneg,plonka_sparse,bit_plon} for reconstructing a vector $\x\in\RR^N$, $N=2^J$, with short support of length $M$, $M$-sparse support or reflected two-block support with block length $M$ from its DFT. In these papers the sought-after vector $\x$ is recovered iteratively from its $2^j$-length periodizations $\x\jj$, where $\x^{(J)}\coloneqq\x$ and $\x\jj$ is obtained by adding the first and second half of $\x\jjj$. However, for the DCT, the concept of periodizations has to be adapted using an iterative application of both reflections and the periodizations from \cite{plonka_smallsupp}, as can be seen in Section \ref{sec:supp}. We still set $\x^{(J)}\coloneqq\x$, but $\x\jj$ is now defined by adding the first half of $\x\jjj$ and the reflection of the second half of $\x\jjj$, i.e.,
\[
 \x\jj\coloneqq\left(x\jjj_0+x\jjj_{2^{j+1}-1},x\jjj_1+x\jjj_{2^{j+1}-2},\dotsc,x\jjj_{2^j-1}+x\jjj_{2^j}\right)^T. 
\]
Employing this concept for $j\in\{L,\dotsc,J-1\}$, where $2^L\geq 2M$, our new algorithm is based on efficiently and iteratively recovering $\x\jjj$ from $\cxx$ using that $\x\jj$ is known. Note that, unlike the DCT reconstruction algorithm for vectors with one-block support in \cite{bit_plon}, which uses a closely related DFT reconstruction and hence complex arithmetic, our algorithm only employs real arithmetic, as it utilizes real factorizations of cosine matrices. This approach requires some observations about the support of $\x\jjj$ if the support of $\x\jj$ is given, which are summarized in Section \ref{sec:supp}. For the reconstruction of $\x\jjj$ from $\x\jj$ we have to distinguish whether the support of $\x\jj$ is contained in its last $M$ entries or whether it is not contained in those entries. In Section \ref{sec:procedures} we present a numerical procedure for each of the two cases. With the help of these methods we develop the sparse fast DCT algorithm for bounded support lengths in Section \ref{sec:alg_bound} and briefly mention a simplified algorithm for exactly known short support lengths in Section \ref{sec:alg_exact}. We conclude our paper by presenting numerical results detailing the performance of our algorithms with respect to runtime and stability for noisy input data in Section \ref{sec:numerics}.

\section{Support Properties of the Reflected Periodizations}\label{sec:supp}
In this paper we want to find a deterministic algorithm for reconstructing $\x\in\RR^N$ with short support of length $m$ from its discrete cosine transform of type II, $\cxx$, if an upper bound $M\geq m$ is known and using, unlike in \cite{bit_plon}, only real arithmetic. In order to do so we adapt techniques used in \cite{plonka_smallsupp,plonka_nonneg,plonka_sparse,bit_plon} for the FFT reconstruction of vectors with short support to the real DCT setting.

There exist several different factorizations of the orthogonal matrix $\Ct_n$, but the following one, see Lemma 2.2 in \cite{plonka_dct}, has proven to be particularly useful in our case. It employs the discrete cosine transform of type IV.
\begin{lem}\label{lem:factorization}
 Let $n\in\NN$ be even and let
 \[
  \PP_n \coloneqq\begin{pmatrix}
                  \left(\delta_{2k,\,l}\right)_{k,\,l=0}^{\frac{n}{2}-1,\,n-1} \\
                  \left(\delta_{2k+1,\,l}\right)_{k,\,l=0}^{\frac{n}{2}-1,\,n-1}
                 \end{pmatrix}
             =\begin{pmatrix}
               1 & 0 & 0 &  0     &   0   & \dots & 0 & 0 \\
               0 & 0 & 1 & 0      &   0   & \dots & 0 & 0 \\
          \vdots &   &   & \vdots &       & \vdots&   & \vdots \\
	       0 &   &   & \dots &       & \dots & 1 & 0 \\
	       0 & 1 & 0 & 0      & 0    & \dots & 0 & 0 \\
	       0 & 0 & 0 & 1      & 0     & \dots & 0 & 0 \\
	  \vdots &   &   & \vdots &       & \vdots&   & \vdots \\
	       0 &\dots&   &\dots&       & \dots & 0 & 1
              \end{pmatrix}\in\RR^{n\times n}
 \]
 be the \emph{even-odd permutation matrix}. Further, define 
 \[
  \T_n\coloneqq\frac{1}{\sqrt{2}}\left(\begin{array}{c|c}
                                       \I_{\frac{n}{2}} & \J_{\frac{n}{2}} \\ \hline
                                       \I_{\frac{n}{2}} & -\J_{\frac{n}{2}}
                                      \end{array}\right)\in\RR^{n\times n},
 \]
 where $\I_{\frac{n}{2}}$ denotes the identity matrix of size $\frac{n}{2}\times\frac{n}{2}$ and
 \[
  \J_{\frac{n}{2}}\coloneqq\left(\delta_{k,\,\frac{n}{2}-1-l}\right)_{k,\,l=0}^{\frac{n}{2}-1} =\begin{pmatrix}
                                                               0 & \dots & 0 & 1 \\
                                                               0 & & 1 & 0 \\
                                                               \vdots & \iddots & & \vdots \\
                                                               1 & \dots & 0 & 0 
                                                              \end{pmatrix}\in\RR^{\frac{n}{2}\times\frac{n}{2}}
 \]
 denotes the \emph{counter identity}. Then $\Ct_n$ satisfies the following factorization,
 \[
  \Ct_n=\PP_n^T \left(\begin{array}{c|c}
               \Ct_{\frac{n}{2}} & \boldsymbol{0}_{\frac{n}{2}}\\ \hline
               \boldsymbol{0}_{\frac{n}{2}} & \Cf_{\frac{n}{2}}
              \end{array}\right)\T_n.
 \]
\end{lem}
From now on let $N\coloneqq2^J$ for $J\geq1$. For $\x\in\RR^{2^{j+1}}$, $j\in\{0,\dotsc,J-1\}$, we denote by
\[
 \x_{(0)}\coloneqq\left(x_k\right)_{k=0}^{2^j-1}\in\RR^{2^j} \quad \text{and}\quad \x_{(1)}\coloneqq\left(x_k\right)_{k=2^j}^{2^{j+1}-1}\in\RR^{2^j}
\]
the first and second half of $\x$, respectively, i.e., $\x^T=\left(\x_{(0)}^T,\x_{(1)}^T\right)$.
\begin{rem}\label{rem:prop}
 Note that for $\x\in\RR^n$, $n$ even, we have
 \begin{equation}\label{eq:permutation}
  \PP_n\x=\begin{pmatrix}
           \left(x_{2k}\right)_{k=0}^{\frac{n}{2}-1} \\
           \left(x_{2k+1}\right)_{k=0}^{\frac{n}{2}-1}
          \end{pmatrix}
 \end{equation}
 and
 \begin{equation}\label{eq:tn}
  \T_n\x=\frac{1}{\sqrt{2}}\begin{pmatrix}
                            \I_{\frac{n}{2}} & \J_{\frac{n}{2}} \\ 
                            \I_{\frac{n}{2}} & -\J_{\frac{n}{2}}
                            \end{pmatrix}   \begin{pmatrix}
                                             \x_{(0)} \\
                                             \x_{(1)}
                                            \end{pmatrix}
                             =\frac{1}{\sqrt{2}}\begin{pmatrix}
                                                 \x_{(0)}+\J_{\frac{n}{2}}\x_{(1)} \\
                                                 \x_{(0)}-\J_{\frac{n}{2}}\x_{(1)}
                                                \end{pmatrix}.
 \end{equation}
\end{rem}
We assume that $\x$ satisfies (\ref{eq:suppose}) in order to guarantee that there is no cancellation of the first and last support entry in the iterative algorithm. Inspired by (\ref{eq:tn}) we define a DCT-II-specific analog to the notion of periodized vectors introduced in \cite{plonka_smallsupp,plonka_nonneg} for DFT algorithms for vectors with short support. Let $\x\in\RR^N$ with $N=2^J$ and set $\x^{(J)}\coloneqq\x$. For $j\in\{0,\dotsc,J-1\}$ define the \emph{reflected periodization} $\x\jj\in \RR^{2^j}$ of $\x$ as
 \begin{equation}\label{eq:periodization}
  \x\jj\coloneqq\x\jjj_{(0)}+\J_{2^{j}}\x\jjj_{(1)}.
 \end{equation}
 We show that the DCT-II of the reflected periodization $\x\jj$ is already completely determined by the DCT-II of $\x$.
\begin{lem}\label{lem:periodization}
 Let $N=2^J$, $J\in\NN$, $\x\in\RR^N$ and $j\in\{0,\dotsc,J\}$. Then
 \[
  \left(\x\jj\right)\dct={\sqrt{2}}^{J-j}\left(\cx_{2^{J-j}k}\right)_{k=0}^{2^j-1}.
 \]
\end{lem}
\begin{proof}
 We prove the lemma by induction. For $j=J$ the claim holds since $\x^{(J)}=\x$. Now we assume the induction hypothesis for some $j\in\{1,\dotsc,J\}$ and show that the claim also holds for $j-1$. It follows from Lemma \ref{lem:factorization}, (\ref{eq:tn}) and the definition of the reflected periodization, (\ref{eq:periodization}), that
 \begin{align}
  \PP_{2^j}\Ct_{2^j}\x\jj&=\frac{1}{\sqrt{2}}\begin{pmatrix}
					      \Ct_{2^{j-1}} & \\
					      & \Cf_{2^{j-1}}
					     \end{pmatrix}\begin{pmatrix}
                                       \I_{2^{j-1}} & \J_{2^{j-1}} \\ 
                                       \I_{2^{j-1}} & -\J_{2^{j-1}}
                                      \end{pmatrix}   \begin{pmatrix}
						       \x\jj_{(0)} \\
						       \x\jj_{(1)}
						      \end{pmatrix} \notag \\
  &=\frac{1}{\sqrt{2}}\begin{pmatrix}
                      \Ct_{2^{j-1}}\x^{(j-1)} \\
                      \Cf_{2^{j-1}}\left(\x\jj_{(0)}-\J_{2^{j-1}}\x\jj_{(1)}\right)
                     \end{pmatrix}; \label{eq:odd}
 \end{align}
 thus motivating the definition of the reflected periodization. Together with the induction hypothesis and (\ref{eq:permutation}) the first $2^{j-1}$ rows of (\ref{eq:odd}) yield that 
 \begin{align*}
  \left(\x^{(j-1)}\right)\dct
  &=\Ct_{2^{j-1}}\x^{(j-1)}
   =\sqrt{2}\left(\PP_{2^j}\Ct_{2^j}\x\jj\right)_{(0)} 
  =\sqrt{2}\left(\left(x\jj\right)\dct_{2k}\right)_{k=0}^{2^{j-1}-1} \\
  &=\sqrt{2}\left(\sqrt{2}^{J-j}\cx_{2^{J-j} 2k}\right)_{k=0}^{2^{j-1}-1} 
  =\sqrt{2}^{J-(j-1)}\left(\cx_{2^{J-(j-1)}k}\right)_{k=0}^{2^{j-1}-1},
 \end{align*}
 which completes the proof.
\end{proof}
Since we always consider vectors $\x\in\RR^{2^J}$ with short support of length $m$ and their reflected periodizations in this paper, we have to introduce some notation for the support of the reflectedly periodized vectors.

 For $j\in\{0,\dotsc,J-1\}$ we say that $\x\jj$ has a \emph{short support of length} $m\jj$ with \emph{first support index} $\mu\jj\in\left\{0,\dotsc,2^j-m\jj\right\}$, \emph{last support index} $\nu\jj\coloneqq\mu\jj+m\jj-1$ and \emph{support interval} $S\jj$ if $x\jj_{\mu\jj},x\jj_{\nu\jj}\neq 0$ and
 \[
  x\jj_k=0 \quad \forall k\notin S\jj\coloneqq I_{\mu\jj,\nu\jj}\coloneqq\left\{\mu\jj,\mu\jj+1,\dotsc,\nu\jj\right\}.
 \] 
 Note that while $S^{(J)}$, i.e., the support interval of $\x=\x^{(J)}$, and $S\jj$ contain all indices at which $\x$ and $\x\jj$, respectively, have nonzero entries, this does not mean that all indices in $S^{(J)}$ and $S\jj$ correspond to nonzero entries, as we require the support sets to be intervals in $\NN_0$ for some of the proofs hereafter. Instead of $m^{(J)}$ we will usually just write $m$.

We can observe the following property of the reflected periodizations.
\begin{lem}\label{lem:support}
 Let $\x\in\RR^{N}$ with $N=2^J$, $J\in\NN$, have a short support of length $m$ and assume that $\x$ satisfies $(\ref{eq:suppose})$. Set $K\coloneqq\left\lceil\log_2m\right\rceil+1$. Then $\x\jj$ has a short support of length $m\jj\leq m$ for all $j\in\{K,\dotsc,J\}$.
\end{lem}
\begin{proof}
We employ an induction argument. By assumption $\x^{(J)}=\x$ has a short support of length $m$. Now suppose that for $j\in\{K,\dotsc,J-1\}$ $\x\jjj$ has a short support of length $m\jjj\leq m$ with support interval $S\jjj=I_{\mu\jjj,\nu\jjj}$, where $\mu\jjj\in\left\{0,\dotsc,2^{j+1}-m\jjj\right\}$ and $\nu\jjj\coloneqq\mu\jjj+m\jjj-1$. We have to distinguish three cases.
 \begin{enumerate}[wide, labelwidth=!, labelindent=0pt]
  \item $S\jjj\subset I_{0,2^j-1}$, i.e., the nonzero entries are contained in the first half of $\x\jjj$. \label{item:supp_i}
  
  Since $\x\jj=\x\jjj_{(0)}+\J_{2^j}\x\jjj_{(1)}$ by (\ref{eq:periodization}), we obtain that $\x\jj$ has a short support with
  \[
   \x\jj=\x\jjj_{(0)} \quad\text{and}\quad S\jj=S\jjj.
  \]
  \item $S\jjj\subset I_{2^j,2^{j+1}-1}$, i.e., the nonzero entries are contained in the second half of $\x\jjj$. \label{item:supp_ii}
  
  The definition of the reflected periodization implies that $\x\jj$ has a short support, as
  \[
   \x\jj=\J_{2^j}\x\jjj_{(1)} \quad\text{and}\quad S\jj=I_{2^{j+1}-1-\nu\jjj,2^{j+1}-1-\mu\jjj}.
  \]
  
  \item $\left\{2^j-1,2^j\right\}\subset S\jjj$ \label{item:supp_iii}
  
  Then at least one possibly nonzero entry from the second half of $\x\jjj$ is added to a possibly nonzero entry from the first half at the reflected index in the computation of $\x\jj$. Thus $\x\jj$ has indeed a short support of length $m\jj<m\jjj$ with support interval
  \begin{align*}
   S\jj=&\left(I_{\mu\jjj,\nu\jjj}\cup I_{2^{j+1}-1-\nu\jjj,2^{j+1}-1-\mu\jjj}\right) \cap I_{0,2^j-1} \\
   \eqcolon&I_{2^j-m\jj,2^j-1}\subsetneq I_{2^j-m\jjj,2^j-1},
  \end{align*}
  and either $\mu\jj=\mu\jjj$ or $\mu\jj=2^{j+1}-1-\nu\jjj$. \qedhere
 \end{enumerate}
\end{proof}
Note that in \ref{item:supp_i} and \ref{item:supp_ii} the support length does not change, i.e., $m\jj=m\jjj$, and that the support length $m\jj<m\jjj$ always decreases in \ref{item:supp_iii}.
\begin{ex}
 \begin{enumerate}[wide, labelwidth=!, labelindent=0pt]
  \item Let $\x\in\RR^{16}$ with nonzero entries $x_{13}$, $x_{14}$, i.e., with short support $S^{(4)}=I_{13,14}$ of length $m=2$. Assume that $m$ is known, i.e., that $M=m=2$. Then $K=2$ and the reflected periodizations $\x\jj$ for $j\in\{K,\dotsc,J\}$ of $\x$ are
  \begin{align*}
   \x=\x^{(4)}&=(0,0,0,0,0,0,0,0,0,0,0,0,0,x_{13},x_{14},0)^T, \\
   \x^{(3)}&=(0,x_{14},x_{13},0,0,0,0,0)^T, \\
   \x^{(2)}&=(0,x_{14},x_{13},0)^T.
  \end{align*}
  Here, $\x^{(3)}$ and $\x^{(2)}$ have the short support $S^{(3)}=S^{(2)}=I_{1,2}$ of length $m^{(3)}=m^{(2)}=m=2$.
  \item Let $\x\in\RR^{16}$ with nonzero entries $x_7$, $x_8$, i.e., with short support $S^{(4)}=I_{7,8}$ of length $m=2$. Again, we assume that $M=m$. Then the reflected periodizations of $\x$ are
  \begin{align*}
   \x=\x^{(4)}&=(0,0,0,0,0,0,0,x_7,x_8,0,0,0,0,0,0,0)^T, \\
   \x^{(3)}&=(0,0,0,0,0,0,0,x_7+x_8)^T, \\
   \x^{(2)}&=(x_7+x_8,0,0,0)^T.
  \end{align*}
  Here, $\x^{(3)}$ has the short support $S^{(3)}=I_{7,7}$ of length $m^{(3)}=1<m=2$ and $\x^{(2)}$ has the short support $S^{(2)}=I_{0,0}$ of length $m^{(2)}=m^{(3)}=1$. \qedhere
 \end{enumerate}
\end{ex}
The aim of our algorithm is to reconstruct $\x$ from $\cxx$ by successively computing its reflected periodizations if only an upper bound $M\geq m$ on the support length of $\x$ is known. Hence, we now investigate the structure of the support of $\x\jjj$ if $\x\jj$ is given. 
\begin{lem}\label{lem:facts}
 Let $\x\in\RR^{N}$ with $N=2^J$, $J\in\NN$, have a short support of length $m\leq M$ and assume that $\x$ satisfies $(\ref{eq:suppose})$. Set $L\coloneqq\left\lceil\log_2M\right\rceil+1$. 
 \begin{enumerate}
  \item \label{item:facts_i} There is at most one index $j'\in\{L,\dotsc,J\}$ such that $S^{(j')}\subset I_{2^{j'}-M,2^{j'}-1}$ and we have that $S^{(j'+1)}\subset I_{2^{j'}-M,2^{j'}+M-1}$ if $j'\leq J-1$.
  \item \label{item:facts_ii} If $j\in\{L,\dotsc,J-1\}\backslash\{j'\}$, then $m\jj=m\jjj$.
  \item \label{item:facts_iii} If $j\in\{L,\dotsc,J-1\}\backslash\{j'\}$ and $S\jj=I_{\mu\jj,\nu\jj}$, then either
 \[
  \x\jjj=\begin{pmatrix}
          \x\jj \\
          \boldsymbol{0}_{2^j}
         \end{pmatrix}
  \quad\text{or}\quad
  \x\jjj=\begin{pmatrix}
          \boldsymbol{0}_{2^j} \\
          \J_{2^j}\x\jj
         \end{pmatrix}
 \]
 with $S\jjj=I_{\mu\jj,\nu\jj}$ or $S\jjj=I_{2^{j+1}-1-\nu\jj,2^{j+1}-1-\mu\jj}$, where $\boldsymbol 0_{2^j}$ denotes the $2^j$-length zero vector.
 \end{enumerate}
\end{lem}
\begin{proof}
\begin{enumerate}[wide, labelwidth=!, labelindent=0pt]
 \item Recall that $K=\left\lceil\log_2m\right\rceil+1\leq L$, so $\x\jj$ has a short support of length $m\jj\leq m$ for all $j\in\{L,\dotsc,J\}$ by Lemma \ref{lem:support}. Set 
 \[
  j'\coloneqq\max\left\{j\in\{L,\dotsc,J\}:S\jj\subset I_{2^j-M,2^j-1}\right\}
 \]
 if such an index exists. First we assume that  there is a $j'\in\{L,\dotsc,J\}$. Then we obtain
 \[
  \x^{(j'-1)}=\underbrace{\x^{(j')}_{(0)}}_{=\boldsymbol{0}_{2^{j'-1}}}+\J_{2^{j'-1}}\x^{(j')}_{(1)} \quad\text{and}\quad 
  S^{(j'-1)}\subset I_{0,M-1}
 \]
 if $j'>L$. Hence,
 \[
  \x\jj=\x\jjj_{(0)}+\J_{2^{j}}\underbrace{\x\jjj_{(1)}}_{=\boldsymbol{0}_{2^{j}}} \quad\text{and}\quad S\jj\subset I_{0,M-1}
 \]
 for all $j\in\{L,\dotsc,j'-2\}$, so $j'$ is the unique index with the above property. Thus, for $j\in\{L,\dotsc,j'-1\}$ the support of $\x\jj$ is contained in the first $M\leq 2^{j-1}$ entries of the vector. By definition of the reflected periodization, (\ref{eq:periodization}), we immediately obtain
 \[
  S^{(j'+1)}\subset I_{2^{j'}-M,2^{j'}+M-1}
 \] 
 if $j'\leq J-1$. For the special case that $m^{(j')}<m^{(j'+1)}$ the supports of the reflected periodizations are depicted in Figure \ref{fig:coll}.
  \def\a{2.4}
 \def\b{\a/4}
 \def\d{0.1}
 \def\dd{3/4*\d}
 \def\c{2/3}
 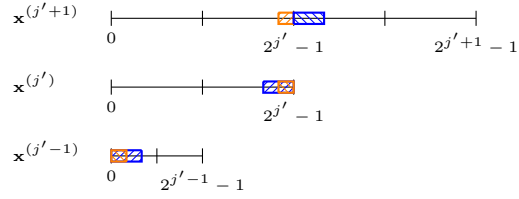
\begin{figure}[!ht]
 \tiny
 \begin{alignat*}{4}
 & \x^{(j'+1)} \quad
 && \begin{tikzpicture}[baseline=-0.75ex]
	 \tikzset{
     brace/.style={decoration={brace},decorate},
     every pin edge/.style={thin}
  }
  \draw (0,0) -- (2*\a,0);
  \draw (0,\d) -- (0,-\d) node[below] {0};
  \draw (2*\a,\d) -- (2*\a,-\d) node[below] {$2^{j'+1}-1$};
  \draw (\a/2,\d) -- (\a/2,-\d);
  \draw (\a,\d) -- (\a,-\d) node[below] {$2^{j'}-1$};
  \draw (3*\a/2,\d) -- (3*\a/2,-\d);
  \draw[orange, thick, pattern=north east lines, pattern color=orange] (\a-\b+\c*\b,3/4*\d) -- (\a-\b+\c*\b,-3/4*\d) -- (\a,-3/4*\d) -- (\a,3/4*\d) -- (\a-\b+\c*\b,3/4*\d);
  \draw[blue, thick, pattern=north west lines, pattern color=blue] (\a,-3/4*\d) -- (\a,3/4*\d) -- (\a+\c*\b,3/4*\d) -- (\a+\c*\b,-3/4*\d) -- (\a,-3/4*\d);
 \end{tikzpicture} \\
 &\x^{(j')} \quad
 && \begin{tikzpicture}[baseline=-0.75ex]
	 \tikzset{
     brace/.style={decoration={brace},decorate},
     every pin edge/.style={thin}
  }
  \draw (0,0) -- (\a,0);
  \draw (0,\d) -- (0,-\d) node[below] {0};
  \draw (\a/2,\d) -- (\a/2,-\d);
  \draw (\a,\d) -- (\a,-\d) node[below] {$2^{j'}-1$};
  \draw[blue, thick, pattern=north east lines, pattern color=blue] (\a-\c*\b,\dd) -- (\a-\c*\b,-\dd) -- (\a,-\dd) -- (\a,\dd) -- (\a-\c*\b,\dd); 
  \draw[orange, thick, pattern=north west lines, pattern color=orange] (\a+\c*\b-\b,-\dd) -- (\a+\c*\b-\b,\dd) -- (\a,\dd) -- (\a,-\dd) -- (\a+\c*\b-\b,-\dd);
 \end{tikzpicture} \\
 &\x^{(j'-1)} \quad
 && \begin{tikzpicture}[baseline=-0.75ex]
	 \tikzset{
     brace/.style={decoration={brace},decorate},
     every pin edge/.style={thin}
  }
  \draw (0,0) -- (\a/2,0);
  \draw (0,\d) -- (0,-\d) node[below] {0};
  \draw (\a/2,\d) -- (\a/2,-\d) node[below] {$2^{j'-1}-1$};
   \draw (\a/4,\d) -- (\a/4,-\d);
  \draw[blue, thick, pattern=north east lines, pattern color=blue] (\c*\b,\dd) -- (\c*\b,-\dd) -- (0,-\dd) -- (0,\dd) -- (\c*\b,\dd); 
  \draw[orange, thick, pattern=north west lines, pattern color=orange] (-\c*\b+\b,-\dd) -- (-\c*\b+\b,\dd) -- (0,\dd) -- (0,-\dd) -- (-\c*\b+\b,-\dd);
 \end{tikzpicture} 
 \end{alignat*}
 \caption{Illustration of the support of $\x^{(j'+1)}$, $\x^{(j')}$ and $\x^{(j'-1)}$ if $m^{(j')}<m^{(j'+1)}$.}
 \label{fig:coll}
 \end{figure}
 \item It follows from Lemma \ref{lem:support} that for decreasing $j$ the support length $m\jj$ cannot increase. Assume that there exists a $j_1\in\{L,\dotsc,J-1\}\backslash\{j'\}$ such that $m^{(j_1)}<m^{(j_1+1)}$. Then case \ref{item:supp_iii} in the proof of Lemma \ref{lem:support} yields that $\left\{2^{j_1}-1,2^{j_1}\right\}\subset S^{(j_1+1)}$. As $m^{(j_1+1)}\leq m\leq M$, this implies that
 \[
  S^{(j_1+1)}\subset I_{2^{j_1}-M,2^{j_1}+M-1},
 \]
 and consequently, by (\ref{eq:periodization}),
 \begin{equation}\label{eq:prop_inclusion}
  S^{(j_1)}\subset I_{2^{j_1}-M,2^{j_1}-1}. 
 \end{equation}
 This is a contradiction, since $j_1\in\{L,\dotsc,J-1\}\backslash\{j'\}$ and $j'$ is, if it exists, the unique index for which (\ref{eq:prop_inclusion}) holds. Hence, we obtain $m\jj=m\jjj$ for all $j\in\{L,\dotsc,J-1\}\backslash\{j'\}$.

 \item For $j\in\{L,\dotsc,J-1\}\backslash\{j'\}$ we have that $m\jj=m\jjj$ by \ref{item:facts_ii}, which also holds if $j'$ does not exist. Hence, the proof of Lemma \ref{lem:support}, cases \ref{item:supp_i} and \ref{item:supp_ii}, shows that either
 \[
  \x\jjj=\begin{pmatrix}
              \x\jj \\
              \boldsymbol 0_{2^j}
             \end{pmatrix} \qquad\text{or}\qquad
  \x\jjj=\begin{pmatrix}
               \boldsymbol 0_{2^j} \\
               \J_{2^j}\x\jj
              \end{pmatrix},
 \]
 as these are the only two $2^{j+1}$-length vectors arising from repeatedly reflectedly periodizing $\x$ that have the reflected periodization $\x\jj$, which can also be seen in Figure \ref{fig:no_coll}. \qedhere
  \def\a{2.4}
 \def\b{\a/4}
 \def\d{0.1}
 \def\dd{3/4*\d}
 \def\c{2/3}
 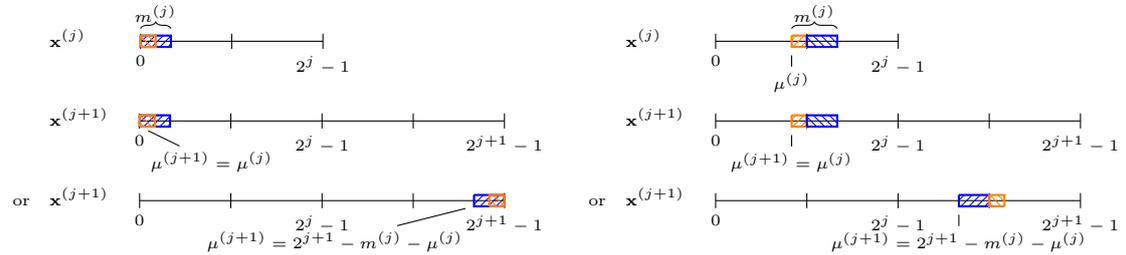
\begin{figure}[!ht]
 \tiny
 \begin{alignat*}{4}
 &\x^{(j)} \quad 
 && \begin{tikzpicture}[baseline=-0.75ex]
	\tikzset{
     brace/.style={decoration={brace},decorate},
     every pin edge/.style={thin}
  }
  \draw (0,0) -- (\a,0);
  \draw (0,\d) -- (0,-\d) node[below] {0};
  \draw (\a,\d) -- (\a,-\d) node[below] {$2^j-1$} ;
  \draw (\a/2,\d) -- (\a/2,-\d) node[below] {};
  \node (1start) at (0,3/4*\d) {};
  \node (1end) at (\c*\b,3/4*\d) {};
  \draw [brace] (1start.north) -- node [pos=0.5, above] {$m^{(j)}$} (1end.north);  
  \draw[blue, thick, pattern=north east lines, pattern color=blue] (0,3/4*\d) -- (0,-3/4*\d) -- (\c*\b,-3/4*\d) -- (\c*\b,3/4*\d) -- (0,3/4*\d);
  \draw[orange, thick, pattern=north west lines, pattern color=orange] (0,-3/4*\d) -- (0,3/4*\d) -- (\b-\c*\b,3/4*\d) -- (\b-\c*\b,-3/4*\d) -- (0,-3/4*\d);
 \end{tikzpicture} 
 \qquad
 &&\x^{(j)} \quad 
 && \begin{tikzpicture}[baseline=-0.75ex]
	\tikzset{
     brace/.style={decoration={brace},decorate},
     every pin edge/.style={thin}
  }
  \draw (0,0) -- (\a,0);
  \draw (0,\d) -- (0,-\d) node[below] {0};
  \draw (\a,\d) -- (\a,-\d) node[below] {$2^j-1$} ;
  \draw (\a/2,\d) -- (\a/2,-\d) node[below] {};
  \node (1start) at (\a/2-\b+\c*\b,3/4*\d) {};
  \node (1end) at (\a/2+\c*\b,3/4*\d) {};
  \draw [brace] (1start.north) -- node [pos=0.5, above] {$m^{(j)}$} (1end.north); 
  \draw[orange, thick, pattern=north east lines, pattern color=orange] (\a/2-\b+\c*\b,3/4*\d) -- (\a/2-\b+\c*\b,-3/4*\d) -- (\a/2,-3/4*\d) -- (\a/2,3/4*\d) -- (\a/2-\b+\c*\b,3/4*\d);
  \draw[blue, thick, pattern=north west lines, pattern color=blue] (\a/2,-3/4*\d) -- (\a/2,3/4*\d) -- (\a/2+\c*\b,3/4*\d) -- (\a/2+\c*\b,-3/4*\d) -- (\a/2,-3/4*\d);
  \node [inner sep=3pt,pin={[inner sep=2pt, pin distance=0.15cm]270:$\mu^{(j)}$}] at (\a/2-\b+\c*\b,-3/4*\d) {};
 \end{tikzpicture} \\
 & \x^{(j+1)} \quad 
 && \begin{tikzpicture}[baseline=-0.75ex]
	 \tikzset{
     brace/.style={decoration={brace},decorate},
     every pin edge/.style={thin}
  }
  \draw (0,0) -- (2*\a,0);
  \draw (0,\d) -- (0,-\d) node[below] {0};
  \draw (2*\a,\d) -- (2*\a,-\d) node[below] {$2^{j+1}-1$};
  \draw (\a/2,\d) -- (\a/2,-\d);
  \draw (\a,\d) -- (\a,-\d) node[below] {$2^j-1$};
  \draw (3*\a/2,\d) -- (3*\a/2,-\d);
  \draw[blue, thick, pattern=north east lines, pattern color=blue] (0,3/4*\d) -- (0,-3/4*\d) -- (\c*\b,-3/4*\d) -- (\c*\b,3/4*\d) -- (0,3/4*\d);
  \draw[orange, thick, pattern=north west lines, pattern color=orange] (0,-3/4*\d) -- (0,3/4*\d) -- (\b-\c*\b,3/4*\d) -- (\b-\c*\b,-3/4*\d) -- (0,-3/4*\d);
  \node [inner sep=3pt,pin={[inner sep=2pt, pin distance=0.15cm]285:$\mu\jjj=\mu\jj$}] at (0,-\dd) {};
 \end{tikzpicture}
 \qquad
 && \x^{(j+1)} \quad
 && \begin{tikzpicture}[baseline=-0.75ex]
	 \tikzset{
     brace/.style={decoration={brace},decorate},
     every pin edge/.style={thin}
  }
  \draw (0,0) -- (2*\a,0);
  \draw (0,\d) -- (0,-\d) node[below] {0};
  \draw (2*\a,\d) -- (2*\a,-\d) node[below] {$2^{j+1}-1$};
  \draw (\a/2,\d) -- (\a/2,-\d);
  \draw (\a,\d) -- (\a,-\d) node[below] {$2^j-1$};
  \draw (3*\a/2,\d) -- (3*\a/2,-\d);
  \draw[orange, thick, pattern=north east lines, pattern color=orange] (\a/2-\b+\c*\b,3/4*\d) -- (\a/2-\b+\c*\b,-3/4*\d) -- (\a/2,-3/4*\d) -- (\a/2,3/4*\d) -- (\a/2-\b+\c*\b,3/4*\d);
  \draw[blue, thick, pattern=north west lines, pattern color=blue] (\a/2,-3/4*\d) -- (\a/2,3/4*\d) -- (\a/2+\c*\b,3/4*\d) -- (\a/2+\c*\b,-3/4*\d) -- (\a/2,-3/4*\d);
  \node [inner sep=3pt,pin={[inner sep=2pt, pin distance=0.15cm]270:$\mu^{(j+1)}=\mu\jj$}] at (\a/2-\b+\c*\b,-3/4*\d) {};
 \end{tikzpicture} \\
  \text{or} \quad &\x^{(j+1)} \quad
 && \begin{tikzpicture}[baseline=-0.75ex]
	 \tikzset{
     brace/.style={decoration={brace},decorate},
     every pin edge/.style={thin}
  }
  \draw (0,0) -- (2*\a,0);
  \draw (0,\d) -- (0,-\d) node[below] {0};
  \draw (2*\a,\d) -- (2*\a,-\d) node[below] {$2^{j+1}-1$};
  \draw (\a/2,\d) -- (\a/2,-\d);
  \draw (\a,\d) -- (\a,-\d) node[below] {$2^j-1$};
  \draw (3*\a/2,\d) -- (3*\a/2,-\d);
  \draw[blue, thick, pattern=north east lines, pattern color=blue] (2*\a-\c*\b,\dd) -- (2*\a-\c*\b,-\dd) -- (2*\a,-\dd) -- (2*\a,\dd) -- (2*\a-\c*\b,\dd); 
  \draw[orange, thick, pattern=north west lines, pattern color=orange] (2*\a+\c*\b-\b,-\dd) -- (2*\a+\c*\b-\b,\dd) -- (2*\a,\dd) -- (2*\a,-\dd) -- (2*\a+\c*\b-\b,-\dd);
  \node [inner sep=3pt,pin={[inner sep=2pt, pin distance=0.15cm]255:$\mu\jjj=2^{j+1}-m\jj-\mu^{(j)}$}] at (2*\a-\c*\b,-\dd) {};
 \end{tikzpicture} 
 \qquad
  \text{or} \quad &&\x^{(j+1)} \quad
 && \begin{tikzpicture}[baseline=-0.75ex]
	 \tikzset{
     brace/.style={decoration={brace},decorate},
     every pin edge/.style={thin}
  }
  \draw (0,0) -- (2*\a,0);
  \draw (0,\d) -- (0,-\d) node[below] {0};
  \draw (2*\a,\d) -- (2*\a,-\d) node[below] {$2^{j+1}-1$};
  \draw (\a/2,\d) -- (\a/2,-\d);
  \draw (\a,\d) -- (\a,-\d) node[below] {$2^j-1$};
  \draw (3*\a/2,\d) -- (3*\a/2,-\d);
  \draw[blue, thick, pattern=north east lines, pattern color=blue] (\a+\a/2-\c*\b,\dd) -- (\a+\a/2-\c*\b,-\dd) -- (3/2*\a,-\dd) -- (3/2*\a,\dd) -- (3/2*\a-\c*\b,\dd); 
  \draw[orange, thick, pattern=north west lines, pattern color=orange] (3/2*\a-\c*\b+\b,-\dd)  -- (3/2*\a-\c*\b+\b,\dd) -- (3/2*\a,\dd) -- (3/2*\a,-\dd) -- (3/2*\a-\c*\b+\b,-\dd);
  \node [inner sep=3pt,pin={[inner sep=2pt, pin distance=0.15cm]270:$\mu^{(j+1)}=2^{j+1}-m\jj-\mu^{(j)}$}] at (3/2*\a-\c*\b,-\dd) {};
 \end{tikzpicture}
 \end{alignat*}
 \caption{Illustration of the two possibilities for the support of $\x^{(j+1)}$ for given $\x^{(j)}$ according to Lemma \ref{lem:facts} for $j\in\{L,\dotsc,j'-1\}$ (left) and $j\in\{j'+1,\dotsc,J-1\}$ (right) with $m^{(j')}<m^{(j'+1)}$.}
 \label{fig:no_coll}
 \end{figure}
 \end{enumerate}
\end{proof}
Lemma \ref{lem:facts} tells us that even if we only know an upper bound $M$ on the support length $m$, there is at most one index $j'$ such that the support of $\x^{(j')}$ is contained in the last $M$ entries. This is also the only case for which the support length of the reflected periodization of double length can increase and for which one might have to undo collisions of nonzero entries in order to compute $\x^{(j'+1)}$ from $\x^{(j')}$. For all other indices the values of the nonzero entries of $\x\jj$ and $\x\jjj$ are the same.

\section{Iterative Sparse DCT Procedures}\label{sec:procedures}
Lemma \ref{lem:periodization} implies that if $\cxx$ is known, the DCTs of all reflected periodizations $\x\jj$ are also known, as they can be obtained by selecting certain entries of $\cxx$. Analogously to \cite{plonka_smallsupp,plonka_nonneg,plonka_sparse,bit_plon}, our goal is to develop an algorithm which recovers $\x\in\RR^{2^J}$ with short support of length $m$ from $\cxx$ by successively calculating the reflected periodizations $\x^{(L)}$, $\x^{(L+1)},\dotsc,\x^{(J)}=\x$ for some starting index $L$ satisfying $m\leq2^{L-1}$. In the following we present both an algorithm for the case that the support length $m$ of $\x$ is known exactly and an algorithm that only requires an upper bound $M\geq m$ on the support length. 

We begin by developing the algorithm for a known bound $M\geq m$ on the support length, which can be easily modified to obtain the algorithm for exactly known support length. Lemma \ref{lem:facts} yields that the values of the nonzero entries and the support lengths of $\x\jj$ and $\x\jjj$ are the same for $j\neq j'$. Hence, if the support of $\x\jj$ is not contained in the last $M$ entries, we only have to find the first support index $\mu\jjj$ of $\x\jjj$, knowing that either $\mu\jjj=\mu\jj$ or $\mu\jjj=2^{j+1}-m\jj-\mu\jj$. However, for $j=j'$, we need to undo the possible collision of nonzero entries from the first and second half of $\x^{(j'+1)}$. 

\subsection{Case 1: No Collision}\label{sec:nocoll_bound}
 If $j\neq j'$, i.e., if $S\jj\not\subset I_{2^j-M,2^j-1}$, then Lemma \ref{lem:facts} implies that for $S\jj=I_{\mu\jj,\nu\jj}$ the values of the nonzero entries of $\x\jj$ and $\x\jjj$ are the same with 
 \[
  m\jjj=m\jj \qquad\text{and}\qquad S\jjj=I_{\mu\jj,\nu\jj} \quad \text{or}\quad S\jjj=I_{2^{j+1}-1-\nu\jj,2^{j+1}-1-\mu\jj}.
 \]
 Hence, we only need to determine whether the first support index is $\mu\jjj=\mu\jj$, i.e., ${\x\jjj}^T=\left({\x\jj}^T,\boldsymbol 0_{2^j}^T\right)$, or $\mu\jjj=2^{j+1}-1-\nu\jj$, i.e., ${\x\jjj}^T=\left(\boldsymbol 0_{2^j}^T,{\J_{2^j}\x\jj}^T\right)$. In order to find out which is the correct first support index, we employ a nonzero entry of $\left(\x\jjj\right)\dct$. First we show how such a nonzero entry can be found efficiently. 
 
 For this we require the notion of the \emph{odd Vandermonde matrix}, which is defined as  
 \[
 \V^{\text{odd}}\left(x_0,\dotsc,x_{n}\right)\coloneqq\left({x_k}^{2l+1}\right)_{k,\,l=0}^{n}
 \]
 for $\left(x_k\right)_{k=0}^n\in\RR^{n+1}$. Recall that the \emph{Vandermonde matrix} 
 \[
  \V\left(x_0,\dotsc,x_n\right)\coloneqq\left({x_k}^l\right)_{k,\,l=0}^n
 \]
 has determinant
 \[
  \det\left(\V\left(x_0,\dotsc,x_n\right)\right)=\prod_{0\leq k<l\leq n}\left(x_l-x_k\right).
 \]
 \begin{lem}\label{lem:odd_vander}
  Let $x_0,\dotsc,x_n\in\RR\backslash\{0\}$ be pairwise distinct such that $\left|x_k\right|\neq\left|x_l\right|$ for all $k\neq l$, where $k,l\in\{0,\dotsc,n\}$. Then the odd Vandermonde matrix $\V^{\mathrm{odd}}\left(x_0,\dotsc,x_n\right)=\left({x_k}^{2l+1}\right)_{k,\,l=0}^n$ is invertible with
  \begin{align*}
   \det\left(\V^{\mathrm{odd}}\left(x_0,\dotsc,x_n\right)\right)
   &=\prod_{j=0}^n x_j\cdot \det\left(\V\left({x_0}^2,\dotsc,{x_n}^2\right)\right) 
   =\prod_{j=0}^n x_j\prod_{0\leq k<l\leq n}\left({x_l}^2-{x_k}^2\right).
  \end{align*}
 \end{lem}
 \begin{proof}
  \begin{align*}
   &\det\left(\V^{\textrm{odd}}\left(x_0,\dotsc,x_n\right)\right) 
   =\det\begin{pmatrix}
         x_0 & {x_0}^3 & {x_0}^5 & \dots & {x_0}^{2n+1} \\
         x_1 & {x_1}^3 & {x_1}^5 & \dots & {x_1}^{2n+1} \\
         \vdots & \vdots & \vdots &      & \vdots       \\
         x_n & {x_n}^3 & {x_n}^5 & \dots & {x_n}^{2n+1}
        \end{pmatrix} \\
   =&\prod_{j=0}^n x_j\cdot\det\begin{pmatrix}
          1 & {x_0}^2 & {x_0}^4 & \dots & {x_0}^{2n} \\
          1 & {x_1}^2 & {x_1}^4 & \dots & {x_1}^{2n} \\
          \vdots & \vdots & \vdots &           & \vdots       \\
          1 & {x_n}^2 & {x_n}^4 & \dots & {x_n}^{2n}  
         \end{pmatrix} 
   =\prod_{j=0}^n {x_j}\cdot\det\left(\V\left({x_0}^2,\dotsc,{x_n}^2\right)\right)  \\
   =&\prod_{j=0}^n x_j \prod_{0\leq k<l\leq n}\left({x_l}^2-{x_k}^2\right).    
  \end{align*}
  As $x_k\neq0$ and $\left|x_k\right|\neq\left|x_l\right|$ for $k\neq l$, $k,l\in\{0,\dotsc,n\}$,  $\V^{\text{odd}}\left(x_0,\dotsc,x_n\right)$ is invertible. 
 \end{proof}
 With the help of odd Vandermonde matrices we can prove the existence of an oddly indexed nonzero entry of $\left(\x\jjj\right)\dct$.
 \begin{lem}\label{lem:nonzero}
  Let $\x\in\RR^N$ with $N=2^J$, $J\in\NN$, have a short support of length $m\leq M$ and assume that $\x$ satisfies $(\ref{eq:suppose})$. Set $L\coloneqq\left\lceil\log_2M\right\rceil+1$. For $j\in\{L,\dotsc,J-1\}\backslash\{j'\}$ let $\x\jj$ be the $2^j$-length reflected periodization of $\x$ with support length $m\jj$. Assume that we have access to all entries of $\cxx$. Then the odd partial vector $\left(\left(x\jjj\right)\dct_{2k+1}\right)_{k=0}^{m\jj-1}$ of $\left(\x\jjj\right)\dct$ has at least one nonzero entry.
 \end{lem}
 \begin{proof}
  We obtain from (\ref{eq:permutation}) and (\ref{eq:odd}) that
  \begin{align}
   \begin{pmatrix}
    \left(\left(x\jjj\right)\dct_{2k}\right)_{k=0}^{2^j-1} \\
    \left(\left(x\jjj\right)\dct_{2k+1}\right)_{k=0}^{2^j-1}
   \end{pmatrix}
    =&\frac{1}{\sqrt{2}}\begin{pmatrix}
                                                           \Ct_{2^j} & \\
                                                           & \Cf_{2^j} 
                                                          \end{pmatrix} 
                                                          \begin{pmatrix}
                                                           \I_{2^j} & \J_{2^j} \\
                                                           \I_{2^j} & -\J_{2^j}
                                                          \end{pmatrix}
                                                          \begin{pmatrix}
                                                           \x\jjj_{(0)} \\
                                                           \x\jjj_{(1)}
                                                          \end{pmatrix} \notag \\
   =&\frac{1}{\sqrt{2}}\begin{pmatrix}
                        \Ct_{2^j} & \\
                        & \Cf_{2^j} 
                       \end{pmatrix}
                       \begin{pmatrix}
                        \x\jj \\
                        \x\jjj_{(0)}-\J_{2^j}\x\jjj_{(1)}
                       \end{pmatrix} \notag \\
   =&\frac{1}{\sqrt{2}}\begin{pmatrix}
                       \left(\x\jj\right)\dct \\
                       \left(2\x\jjj_{(0)}-\x\jj\right)\dctf
                      \end{pmatrix}, \label{eq:even_odd}
  \end{align}
  where we used that $\J_{2^j}\x\jjj_{(1)}=\x\jj-\x\jjj_{(0)}$ by (\ref{eq:periodization}). If we denote the support interval of $\x\jjj_{(0)}$ by $S\jjj_{(0)}$, Lemma \ref{lem:facts} yields that $S\jjj_{(0)}=S\jj$ or $S\jjj_{(0)}=\emptyset$, since $j\neq j'$. Consequently, $S\jjj_{(0)}\subset S\jj$ and $S\left(2\x\jjj_{(0)}-\x\jj\right)\subset S\jj$, where $S(\y)$ is the support interval of $\y\in\RR^n$. As $\left|S\jj\right|=m\jj\leq m\leq M$, we can restrict (\ref{eq:even_odd}) to the rows corresponding to the first $m\jj$ oddly indexed entries of $\left(\x\jjj\right)\dct$ and find
  \begin{align}
   &\left(\left(x\jjj\right)\dct_{2k+1}\right)_{k=0}^{m\jj-1} 
   =\frac{1}{\sqrt{2}}\left(\left(\Cf_{2^j}\right)_{k,\,l}\right)_{k,\,l=0}^{m\jj-1,\,2^j-1} \left(2\x\jjj_{(0)}-\x\jj\right) \notag \\
   =&\frac{1}{\sqrt{2^j}}\left(\sum_{l\in S\jj}\cos\left(\frac{(2k+1)(2l+1)\pi}{4\cdot 2^j}\right)\left(2\x\jjj_{(0)}-\x\jj\right)_l\right)_{k=0}^{m\jj-1} \notag \\
   \eqcolon&\frac{1}{\sqrt{2^j}}\cdot\T\jj\cdot\left(\left(2\x\jjj_{(0)}-\x\jj\right)_l\right)_{l\in S\jj}. \label{eq:cheb_mat}
  \end{align}
  Note that $\T\jj$ is the restriction of the cosine matrix of type IV without the normalization factor to the first $m\jj$ rows and the $m\jj$ columns indexed by $S\jj$. We show that $\T\jj$ is invertible, using Chebyshev polynomials.  
  For $x\in\RR$ with $|x|\leq 1$ and $n\in\NN_0$ the \emph{Chebyshev polynomial of the first kind of degree} $n$ is defined as
  \[
   T_n(x)\coloneqq\cos(n\arccos x) \eqqcolon\sum_{l=0}^n a_{n,l}x^l.
  \]
  Note that the leading coefficient of $T_n$ satisfies
  \begin{equation}\label{eq:cheb_lead}
   a_{n,n}=\begin{cases}
         1 & \text{if }n=0, \\
         2^{n-1} & \text{if } n\geq 1,
        \end{cases}
  \end{equation}
  and that $T_n$ is odd if $n$ is odd, and $T_n$ is even if $n$ is even.
  
  Further, for $n\in\NN$, we define the \emph{Chebyshev zero nodes} 
  \[
   t_{l,n}\coloneqq\cos\left(\frac{(2l+1)\pi}{2n}\right), \qquad l\in\{0,\dotsc,n-1\},
  \]
  which are exactly the $n$ zeros of the $n$th Chebyshev polynomial of the first kind. Then
  \begin{equation}\label{eq:cheb_eval}
   T_k\left(t_{l,n}\right)=\cos\left(\frac{k(2l+1)\pi}{2n}\right)
  \end{equation}
  for all $l\in\{0,\dotsc,n-1\}$, $n\in\NN$ and $k\in\NN_0$, since  $\left|t_{l,n}\right|\leq 1$. Using (\ref{eq:cheb_eval}), the coefficient representation of the Chebyshev polynomials and the fact that $a_{2k+1,2l}=0$ for all $l\in\{0,\dotsc,k\}$ and $k\in\NN_0$, we find for $\T\jj$ that
  \begin{align}
   &\T\jj=\left(\cos\left(\frac{(2k+1)(2l+1)\pi}{2\cdot2^{j+1}}\right)\right)_{k=0,\,l\in S\jj}^{m\jj-1}
   =\left(T_{2k+1}\left(t_{l,2^{j+1}}\right)\right)_{k=0,\,l\in S\jj}^{m\jj-1} \notag \\
   =&\left(\sum_{\substack{r'=0 \\ r'\equiv 1\mods 2}}^{2k+1} a_{2k+1,r'}\cdot t_{l,2^{j+1}}^{r'}\right)_{k=0,\,l\in S\jj}^{m\jj-1} 
   =\left(a_{2k+1,2r+1}\right)_{k,\,r=0}^{m\jj-1}\cdot \left(t_{l,2^{j+1}}^{2r+1}\right)_{r=0,\,l\in S\jj}^{m\jj-1} \label{eq:cheb_coeff}\\
   =&\begin{pmatrix}
      a_{11} & 0      & 0      & \dots & 0 \\
      a_{31} & a_{33} & 0      & \dots & 0 \\
      \vdots & \vdots & \vdots &       & 0 \\
      a_{2m\jj-1,1} & a_{2m\jj-1,3} & a_{2m\jj-1,5} & \dots & a_{2m\jj-1,2m\jj-1}
     \end{pmatrix}
     \begin{pmatrix}
      \left(t_{l,2^{j+1}}\right)_{l\in S\jj}^T \\
      \left({t_{l,2^{j+1}}}^3\right)_{l\in S\jj}^T \\
      \vdots \\
      \left({t_{l,2^{j+1}}}^{2m\jj-1}\right)_{l\in S\jj}^T \\
     \end{pmatrix} \notag \\
     \eqcolon& \A\jj\cdot \V^{\mathrm{odd}}\left(\left(t_{l,2^{j+1}}\right)_{l\in S\jj}\right)^T, \label{eq:matrix_prod}
  \end{align}
  where we set $a_{2k+1,2r+1}\coloneqq0$ for $r\in\left\{k+1,\dotsc,m\jj-1\right\}$ in (\ref{eq:cheb_coeff}). By (\ref{eq:cheb_lead}) the triangular matrix $\A\jj$ is invertible. Furthermore, since $S\jj\subset I_{0,2^j-1}$, 
  \[
   \frac{(2l+1)\pi}{2\cdot 2^{j+1}}\in \left(0,\frac{\pi}{2}\right)
  \]
  for all $l\in S\jj$. Consequently, we have that
  \[
   t_{l,2^{j+1}}=\cos\left(\frac{(2l+1)\pi}{2\cdot 2^{j+1}}\right)\in (0,1),
  \]
  and $\left|t_{k,2^{j+1}}\right|\neq\left|t_{l,2^{j+1}}\right|$ for all $k\neq l$, $k,l\in S\jj$, as the cosine is bijective on $\left(0,\frac{\pi}{2}\right)$. Hence $\V^{\mathrm{odd}}\left(\left(t_{l,2^{j+1}}\right)_{l\in S\jj}\right)^T$ is invertible by Lemma \ref{lem:odd_vander}, so $\T\jj$ is invertible as well. Assume now that $\left(x\jjj\right)_{2k+1}\dct=0$ for all $k\in\left\{0,\dotsc,m\jj-1\right\}$. Then  (\ref{eq:cheb_mat}) and (\ref{eq:matrix_prod}) yield
  \begin{align}
   &\boldsymbol{0}_{m\jj}=\left(\left(x\jjj\right)_{2k+1}\dct\right)_{k=0}^{m\jj-1}=\frac{1}{\sqrt{2^j}}\T\jj\left(\left(2\x\jjj_{(0)}-\x\jj\right)_l\right)_{l\in S\jj} \notag \\
   \Leftrightarrow \quad &\boldsymbol{0}_{m\jj}=\left(\left(2\x\jjj_{(0)}-\x\jj\right)_l\right)_{l\in S\jj}. \label{eq:contradiction}
  \end{align}
  However, since $j\neq j'$, we have that $\x\jjj_{(0)}=\x\jj$ and $\x\jjj_{(1)}=\boldsymbol{0}_{2^j}$, or $\x\jjj_{(0)}=\boldsymbol{0}_{2^j}$ and $\x\jjj_{(1)}=\J_{2^j}\x\jj$. In either case (\ref{eq:contradiction}) is only possible if $\x\jj=\boldsymbol{0}_{2^j}$, which is a contradiction to (\ref{eq:suppose}) and the fact that $\x\neq\boldsymbol 0_{N}$ has a short support of length $m$. Hence, there exists an index $k_0\in\left\{0,\dotsc,m\jj-1\right\}$ such that $\left(x\jjj\right)\dct_{2k_0+1}\neq 0$.
  
  For the implementation of this procedure, using Lemma \ref{lem:periodization}, set
  \[
   k_0\coloneqq\underset{k\in\left\{0,\dotsc,m\jj-1\right\}}{\argmax}\left\{\left|\sqrt{2}^{J-j-1}\cx_{2^{J-j-1}(2k+1)}\right|\right\}.
  \]
  Then $\left(x\jjj\right)\dct_{2k_0+1}\neq0$ and it is likely that this entry is not too close to zero, which is supported empirically by the numerical experiments in Section \ref{sec:numerics}.
 \end{proof}

 Now we show how $\x\jjj$ can be computed from $\x\jj$ and one oddly indexed nonzero entry of $\left(\x\jjj\right)\dct$ using the following theorem.
 \begin{thm}\label{thm:nocoll}
  Let $\x\in\RR^{N}$ with $N=2^J$, $J\in\NN$, have a short support of length $m\leq M$ and assume that $\x$ satisfies $(\ref{eq:suppose})$. Set $L\coloneqq\left\lceil\log_2M\right\rceil+1$. For $j\in\{L,\dotsc,J-1\}\backslash\{j'\}$ let $\x\jj$ be the $2^j$-length reflected periodization of $\x$ with support length $m\jj$. Assume that we have access to all entries of $\cxx$. Then $\x\jjj$ can be uniquely recovered from $\x\jj$ and one nonzero entry of $\left(\sqrt{2}^{J-j-1}\cx_{2^{J-j-1}(2k+1)}\right)_{k=0}^{m\jj-1}$, using $\mathcal{O}\left(m\jj\right)$ operations.
 \end{thm}
 \begin{proof}
  By Lemma \ref{lem:facts} \ref{item:facts_iii} there are precisely two vectors in $\RR^{2^{j+1}}$ that arise from reflectedly periodizing $\x$ and have the given reflected periodization $\x\jj$, namely
  \[
   \uu^0\coloneqq\begin{pmatrix}
                 \x\jj \\
                 \boldsymbol{0}_{2^j}
                \end{pmatrix}
   \quad\text{and}\quad 
   \uu^1\coloneqq\begin{pmatrix}
                 \boldsymbol{0}_{2^j} \\
                 \J_{2^j}\x\jj
                \end{pmatrix}.
  \]
  Assuming that $S\jj=I_{\mu\jj,\nu\jj}$, $\uu^0$ has the first support index $\mu\jj$, $\uu^1$ has the first support index $2^{j+1}-m\jj-\mu\jj$ and both have a support of length $m\jjj=m\jj$. Let us now compare the DCTs of $\uu^0$ and $\uu^1$. Lemma \ref{lem:factorization} yields
  \begin{align*}
   \begin{pmatrix}
    \left(\left(u^0\right)\dct_{2k}\right)_{k=0}^{2^j-1} \\
    \left(\left(u^0\right)\dct_{2k+1}\right)_{k=0}^{2^j-1}
   \end{pmatrix}
   =&\PP_{2^{j+1}}\left(\uu^0\right)\dct=\frac{1}{\sqrt{2}}\begin{pmatrix}
                                                           \Ct_{2^j} & \\
                                                           & \Cf_{2^j} 
                                                          \end{pmatrix} 
                                                          \begin{pmatrix}
                                                           \I_{2^j} & \J_{2^j} \\
                                                           \I_{2^j} & -\J_{2^j}
                                                          \end{pmatrix}
                                                          \begin{pmatrix}
                                                           \x\jj \\
                                                           \boldsymbol{0}_{2^j}
                                                          \end{pmatrix} \\
   =&\frac{1}{\sqrt{2}}\begin{pmatrix}
                        \Ct_{2^j} & \\
                        & \Cf_{2^j} 
                       \end{pmatrix}
                       \begin{pmatrix}
                        \x\jj \\
                        \x\jj
                       \end{pmatrix}
   =\frac{1}{\sqrt{2}}\begin{pmatrix}
                       \left(\x\jj\right)\dct \\
                       \left(\x\jj\right)^{\widehat{\mathrm{IV}}}
                      \end{pmatrix}
  \end{align*}
  and 
  \begin{align*}
   \begin{pmatrix}
    \left(\left(u^1\right)\dct_{2k}\right)_{k=0}^{2^j-1} \\
    \left(\left(u^1\right)\dct_{2k+1}\right)_{k=0}^{2^j-1}
   \end{pmatrix}
    =&\frac{1}{\sqrt{2}}\begin{pmatrix}
                                                           \Ct_{2^j} & \\
                                                           & \Cf_{2^j} 
                                                          \end{pmatrix} 
                                                          \begin{pmatrix}
                                                           \I_{2^j} & \J_{2^j} \\
                                                           \I_{2^j} & -\J_{2^j}
                                                          \end{pmatrix}
                                                          \begin{pmatrix}
                                                           \boldsymbol{0}_{2^j} \\
                                                           \J_{2^j}\x\jj
                                                          \end{pmatrix} \\
   =&\frac{1}{\sqrt{2}}\begin{pmatrix}
                        \Ct_{2^j} & \\
                        & \Cf_{2^j} 
                       \end{pmatrix}
                       \begin{pmatrix}
                        \J_{2^j}\left(\J_{2^j}\x\jj\right) \\
                        -\J_{2^j}\left(\J_{2^j}\x\jj\right)
                       \end{pmatrix}
   =\frac{1}{\sqrt{2}}\begin{pmatrix}
                       \left(\x\jj\right)\dct \\
                       -\left(\x\jj\right)^{\widehat{\mathrm{IV}}}
                      \end{pmatrix}.
  \end{align*}
  Consequently, we have that
  \begin{equation}\label{eq:u}
   \left(u^1\right)\dct_{2k+1}=-\left(u^0\right)\dct_{2k+1}, \quad k\in\left\{0,\dotsc,2^j-1\right\},
  \end{equation}
  for all oddly indexed entries of $\left(\uu^0\right)\dct$ and $\left(\uu^1\right)\dct$. In order to decide whether $\x\jjj=\uu^0$ or $\x\jjj=\uu^1$ we compare a nonzero entry $\left(x\jjj\right)\dct_{2k_0+1}=\sqrt{2}^{J-j-1}\cx_{2^{J-j-1}(2k_0+1)}\neq 0$ to the corresponding entry of $\uu^0$. By Lemma \ref{lem:nonzero} $\left(x\jjj\right)\dct_{2k_0+1}$ can be found by examining $m\jj$ entries of $\cxx$. If $\left(u^0\right)\dct_{2k_0+1}=\left(x\jjj\right)\dct_{2k_0+1}$, then $\x\jjj=\uu^0$ by (\ref{eq:u}), and if $\left(u^0\right)\dct_{2k_0+1}=-\left(x\jjj\right)\dct_{2k_0+1}$, then $\x\jjj=\uu^1$. Numerically, we set $\x\jjj=\uu^0$ if
  \[
   \left|\left(u^0\right)\dct_{2 k_0+1}-\sqrt 2^{J-j-1}\cx_{2^{J-j-1}(2 k_0+1)}\right|<\left|\left(u^0\right)\dct_{2 k_0+1}+\sqrt 2^{J-j-1}\cx_{2^{J-j-1}(2 k_0+1)}\right|
  \]
  and $\x\jjj=\uu^1$ otherwise. The required entry of $\uu^0$ can be computed from $\x\jj$ using $\mathcal{O}\left(m\jj\right)=\mathcal{O}(m)$ operations,
  \[
   \left(u^0\right)\dct_{2 k_0+1}=\sum_{l=0}^{2^{j+1}-1}\left(\Ct_{2^{j+1}}\right)_{2 k_0+1,\,l}u^0_l
   =\sum_{l=0}^{m\jj-1}\left(\Ct_{2^{j+1}}\right)_{2 k_0+1,\,\mu\jj+l}x\jj_{\mu\jj+l}.
  \]
  Thus we can find the first support index $\mu\jjj$ via
  \[
   \mu\jjj\coloneqq\begin{cases}
            \mu\jj &\text{if } \x\jjj=\uu^0, \\
            2^{j+1}-m\jj-\mu\jj & \text{if } \x\jjj=\uu^1.
           \end{cases}\qedhere
  \]
 \end{proof}
 \subsection{Case 2: Possible Collision}\label{sec:coll_bound}
 If $j=j'$, i.e., if $S\jj\subset I_{2^j-M,2^j-1}$, Lemma \ref{lem:facts} yields that $S\jjj\subset I_{2^j-M,2^j+M-1}$ and that nonzero entries of $\x\jjj$ might have been added to obtain $\x\jj$, so the values of the  nonzero entries of $\x\jj$ and $\x\jjj$ are not necessarily the same. The support of $\x\jj$ has length $m\jj\leq m\leq M$, so, by definition of the reflected periodization and Lemma \ref{lem:facts}, the support of $\x\jjj_{(0)}$ has at most length $\widetilde m\jj\coloneqq2^j-\mu\jj\leq M$. Note that $\widetilde m\jj\geq m\jj$ and that $\widetilde m\jj>m\jj$ is possible if there is no collision, i.e., if $2^j-1\notin S\jj$, see Figure \ref{fig:poss_coll}. Hence, it suffices to consider restrictions of $\x\jj$ and $\x\jjj_{(0)}$ to vectors of length $2^{\tilde K-1}$, where $2^{\tilde K-2}<\widetilde m\jj\leq 2^{\tilde K-1}$, taking into account all of their relevant entries.
 \def\a{2.4}
 \def\b{\a/4}
 \def\d{0.1}
 \def\dd{3/4*\d}
 \def\c{2/3}
 \begin{figure}[!ht]
 \tiny
 \begin{alignat*}{4}
 &\x^{(j)} \quad
 && \begin{tikzpicture}[baseline=-0.75ex]
	 \tikzset{
     brace/.style={decoration={brace},decorate},
     every pin edge/.style={thin}
  }
  \draw (0,0) -- (\a,0);
  \draw (0,5/4*\d) -- (0,-5/4*\d) node[below] {0}; 
  \draw (0.6*\a,\d) -- (0.6*\a,-\d); 
  \draw (\a/2,5/4*\d) -- (\a/2,-5/4*\d);
  \draw (\a,5/4*\d) -- (\a,-5/4*\d);
  \node (1end) at (\a,3/4*\d) {};
  \node (1start) at (\a-\c*\b,3/4*\d) {};
  \draw [brace] (1start.north) -- node [pos=0.5, above] {$m^{(j)}$} (1end.north); 
  \draw[blue, thick, pattern=north east lines, pattern color=blue] (\a,\dd) -- (\a,-\dd) -- (\a-\c*\b,-\dd) -- (\a-\c*\b,\dd) -- (\a,\dd); 
  \draw[orange, thick, pattern=north east lines, pattern color=orange] (\a,\dd) -- (\a,-\dd) -- (\a-\b+\c*\b,-\dd) -- (\a-\b+\c*\b,\dd) -- (\a,\dd); 
  \node [inner sep=3pt,pin={[inner sep=2pt, pin distance=0.15cm]255:$2^{j}-M$}] at (0.6*\a,-3/4*\d) {};
  \node [inner sep=3pt,pin={[inner sep=2pt, pin distance=0.15cm]270:$\mu\jj$}] at (\a-\c*\b,-3/4*\d) {};
  \node [inner sep=3pt,pin={[inner sep=2pt, pin distance=0.15cm]275:$2^{j}-1$}] at (\a,-3/4*\d) {};
 \end{tikzpicture} 
 \qquad
 &&\x^{(j)} \quad 
 && \begin{tikzpicture}[baseline=-0.75ex]
	 \tikzset{
     brace/.style={decoration={brace},decorate},
     every pin edge/.style={thin}
  }
  \draw (0,0) -- (\a,0);
  \draw (0,5/4*\d) -- (0,-5/4*\d) node[below] {0}; 
  \draw (0.6*\a,\d) -- (0.6*\a,-\d); 
  \draw (\a/2,5/4*\d) -- (\a/2,-5/4*\d);
  \draw (\a,5/4*\d) -- (\a,-5/4*\d);
  \node (1end) at (0.7*\a+\c*\b,3/4*\d) {};
  \node (1start) at (0.7*\a,3/4*\d) {};
  \draw [brace] (1start.north) -- node [pos=0.5, above] {$m^{(j)}$} (1end.north); 
  \draw[blue, thick, pattern=north east lines, pattern color=blue] (0.7*\a,\dd) -- (0.7*\a,-\dd) -- (0.7*\a+\c*\b,-\dd) -- (0.7*\a+\c*\b,\dd) -- (0.7*\a,\dd);
  \node [inner sep=3pt,pin={[inner sep=2pt, pin distance=0.15cm]255:$2^{j}-M$}] at (0.6*\a,-3/4*\d) {};
  \node [inner sep=3pt,pin={[inner sep=2pt, pin distance=0.15cm]270:$\mu\jj$}] at (0.7*\a,-3/4*\d) {};
  \node [inner sep=3pt,pin={[inner sep=2pt, pin distance=0.15cm]275:$2^{j}-1$}] at (\a,-3/4*\d) {};
 \end{tikzpicture} \\
  & \x^{(j+1)} \quad
 && \begin{tikzpicture}[baseline=-0.75ex]
	 \tikzset{
     brace/.style={decoration={brace},decorate},
     every pin edge/.style={thin}
  }
  \draw (0,0) -- (2*\a,0);
  \draw (0,5/4*\d) -- (0,-5/4*\d) node[below] {0};
  \draw (2*\a,5/4*\d) -- (2*\a,-5/4*\d) node[below] {$2^{j+1}-1$};
  \draw (\a/2,5/4*\d) -- (\a/2,-5/4*\d);
  \draw (0.6*\a,\d) -- (0.6*\a,-\d);
  \draw (\a,5/4*\d) -- (\a,-5/4*\d) node[below] {$2^{j}-1$};
  \draw (3*\a/2,5/4*\d) -- (3*\a/2,-5/4*\d);
  \draw (\a-\b*\c,\d) -- (\a-\b*\c,-\d);  
  \node (1end) at (\a,3/4*\d) {};
  \node (1start) at (\a-\c*\b,3/4*\d) {};
  \draw [brace] (1start.north) -- node [pos=0.5, above] {$\widetilde m^{(j)}$} (1end.north);
  \draw[orange, thick, pattern=north east lines, pattern color=orange] (\a-\b+\c*\b,3/4*\d) -- (\a-\b+\c*\b,-3/4*\d) -- (\a,-3/4*\d) -- (\a,3/4*\d) -- (\a-\b+\c*\b,3/4*\d);
  \draw[blue, thick, pattern=north west lines, pattern color=blue] (\a,-3/4*\d) -- (\a,3/4*\d) -- (\a+\c*\b,3/4*\d) -- (\a+\c*\b,-3/4*\d) -- (\a,-3/4*\d);
  \node [inner sep=3pt,pin={[inner sep=2pt, pin distance=0.15cm]255:$\mu\jj$}] at (\a-\b*\c,-3/4*\d) {};
  \node [inner sep=3pt,pin={[inner sep=2pt, pin distance=0.15cm]255:$2^{j}-M$}] at (0.6*\a,-3/4*\d) {};
 \end{tikzpicture}
 \qquad
  && \x^{(j+1)} \quad
  && \begin{tikzpicture}[baseline=-0.75ex]
	 \tikzset{
     brace/.style={decoration={brace},decorate},
     every pin edge/.style={thin}
  }
  \draw (0,0) -- (2*\a,0);
  \draw (0,5/4*\d) -- (0,-5/4*\d) node[below] {0};
  \draw (2*\a,5/4*\d) -- (2*\a,-5/4*\d) node[below] {$2^{j+1}-1$};
  \draw (\a/2,5/4*\d) -- (\a/2,-5/4*\d);
  \draw (0.6*\a,\d) -- (0.6*\a,-\d);
  \draw (\a,5/4*\d) -- (\a,-5/4*\d) node[below] {$2^{j}-1$};
  \draw (3*\a/2,5/4*\d) -- (3*\a/2,-5/4*\d);
  \node (1end) at (\a,3/4*\d) {};
  \node (1start) at (0.7*\a,3/4*\d) {};
  \draw [brace] (1start.north) -- node [pos=0.5, above] {$\widetilde m^{(j)}$} (1end.north);
  \draw[blue, thick, pattern=north west lines, pattern color=blue] (0.7*\a,-3/4*\d) -- (0.7*\a,3/4*\d) -- (0.7*\a+\c*\b,3/4*\d) -- (0.7*\a+\c*\b,-3/4*\d) -- (0.7*\a,-3/4*\d);
  \node [inner sep=3pt,pin={[inner sep=2pt, pin distance=0.15cm]270:$\mu\jj$}] at (0.7*\a,-3/4*\d) {};
  \node [inner sep=3pt,pin={[inner sep=2pt, pin distance=0.15cm]255:$2^{j}-M$}] at (0.6*\a,-3/4*\d) {};
 \end{tikzpicture}
  \end{alignat*}
 \caption{Illustration of the support of $\x^{(j)}$ and one possibility for the support of $\x\jjj$ for $m^{(j)}<m^{(j+1)}$ (left) and for $m^{(j)}=m^{(j+1)}$ (right), with $j=j'$.}
 \label{fig:poss_coll}
 \end{figure}
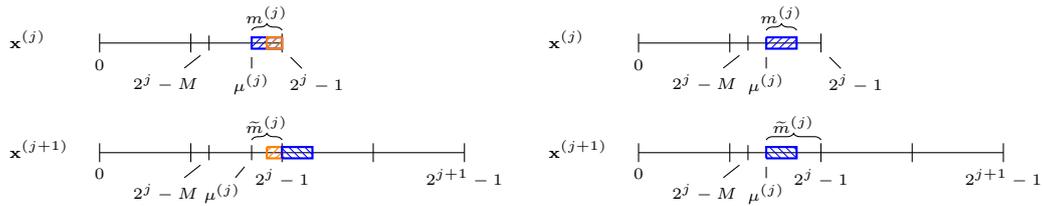
 \newline
 We then show that $\x\jjj$ can be calculated using essentially one DCT of length $2^{\tilde K-1}$ and further operations of complexity $\mathcal{O}\left(2^{\tilde K}\right)$. In order to do this we have to employ the vector $\x\jj$ known from the previous iteration step and $2^{\tilde K}$ suitably chosen oddly indexed entries of $\left(\x\jjj\right)\dct$, which can be found from $\cxx$ by Lemma \ref{lem:periodization}.
 
 The efficient computation of $\x\jjj$ is based on the following theorem.
 \begin{thm}\label{thm:coll}
  Let $\x\in\RR^{N}$ with $N=2^J$, $J\in\NN$, have a short support of length $m\leq M$ and assume that $\x$ satisfies $(\ref{eq:suppose})$. Let $j=j'$ and $\x^{(j)}$ be the $2^{j}$-length reflected periodization of $\x$ according to $(\ref{eq:periodization})$ and Lemma $\ref{lem:facts}$ with first support index $\mu\jj$ and support length $m\jj$. Assume that we have access to all entries of $\cxx$. Set $\widetilde m\jj\coloneqq 2^j-\mu\jj$, $\tilde K\coloneqq\left\lceil\log_2\widetilde m\jj\right\rceil+1$ and define the restrictions of $\x\jj$,  $\x\jjj_{(0)}$ and $\x\jjj_{(1)}$ to $2^{\tilde K-1}$-length vectors 
  \begin{align*}
   \z\jj&\coloneqq\left(x\jj_k\right)_{k=2^j-2^{\tilde K-1}}^{2^j-1} \\
   \z\jjj_{(0)}&\coloneqq\left(x\jjj_k\right)_{k=2^j-2^{\tilde K-1}}^{2^j-1}  \quad\text{and}\quad \z\jjj_{(1)}\coloneqq\left(x\jjj_k\right)_{k=2^j}^{2^j+2^{\tilde K-1}-1}.
  \end{align*}
  Then, using the vectors of samples $\bb^0\coloneqq\sqrt{2}^{J-j-1}\left(\cx_{2^{J-j-1}(2\cdot2^{j-\tilde K}(2p+1)+1)}\right)_{p=0}^{2^{\tilde K-1}-1}$ and $\bb^1\coloneqq\sqrt{2}^{J-j-1}\left(\cx_{2^{J-j-1}(2(2^{j-\tilde K}(2p+1)-1)+1)}\right)_{p=0}^{2^{\tilde K-1}-1}$, it holds that
  \begin{align*}
   \z\jjj_{(0)}&=\frac{1}{2}\left(\sqrt{2^{j-\tilde K}}(-1)^{2^{j-\tilde K}}\J_{2^{\tilde K-1}}\diag(\tilde\cc)\D_{2^{\tilde K-1}}\Cf_{2^{\tilde K-1}}\J_{2^{\tilde K-1}}\left(\bb^0-\bb^1\right)+\z\jj\right) \quad\text{and} \\
   \z\jjj_{(1)}&=\J_{2^{\tilde K-1}}\left(\z\jj-\z\jjj_{(0)}\right),
  \end{align*}
  where $\D_{2^{\tilde K-1}}\coloneqq\diag\left((-1)^k\right)_{k=0}^{2^{\tilde K-1}-1}$ and $\tilde\cc\coloneqq\left(\cos\left(\frac{(2k+1)\pi}{4\cdot 2^j}\right)^{-1}\right)_{k=0}^{2^{\tilde K-1}-1}$. The reflected periodization $\x\jjj$ is given as
  \[
   x\jjj_k=\begin{cases}
            \left(z\jjj_{(0)}\right)_{k-2^j+2^{\tilde K-1}} & \text{if } k\in\{2^j-2^{\tilde K-1},\dotsc,2^j-1\}, \\
            \left(z\jjj_{(1)}\right)_{k-2^j} & \text{if } k\in\{2^j,\dotsc,2^j+2^{\tilde K-1}-1\}, \\
            0 & \text{else.}
           \end{cases}
  \] 
 \end{thm}
 \begin{proof}
  If $j=j'$, it follows from Lemmas \ref{lem:support} and \ref{lem:facts} for the support set $S\jj$ of $\x\jj$ that 
  \[
   S\jj\subset I_{\mu\jj,2^j-1}\subset I_{2^j-M,2^j-1}.
  \]
  With $\widetilde m\jj=2^j-\mu\jj\leq M$ and $\tilde K=\left\lceil\log_2\widetilde m\jj\right\rceil+1$, we obtain
  \[
   S\jj\subset I_{2^j-\widetilde m\jj,2^j-1}\subset I_{2^j-2^{\tilde K-1},2^j-1},
  \]
  and, by definition of the reflected periodization, the support set $S\jjj$ of $\x\jjj$ satisfies
  \[
   S\jjj\subset I_{2^j-\widetilde m\jj,2^j+\widetilde m\jj-1}\subset I_{2^j-2^{\tilde K-1},2^j+2^{\tilde K-1}-1}.
  \]
  This allows us to reduce the number computations necessary to find $\x\jjj$. Note that since we only suppose that $x_{\mu^{(J)}}\neq0$, $x_{\nu^{(J)}}\neq0$ and $x_{\mu^{(J)}}+x_{\nu^{(J)}}\neq 0$ in (\ref{eq:suppose}), some of the last $\widetilde m\jj$ entries of $\x\jj$ might be zero, despite being obtained by adding two nonzero entries of $\x\jjj$. However, we know that either $\mu\jj=\mu\jjj$ or $\mu\jj=2^{j+1}-1-\nu\jj$ by case \ref{item:supp_iii} in the proof of Lemma \ref{lem:support}. Hence, if we restrict $\x\jj$ to its last $2^{\tilde K-1}\geq\widetilde m\jj=2^j-\mu\jj$ entries, i.e., to $\z\jj$, we take all of the at most $\widetilde m\jj$ entries into account which correspond to possibly nonzero entries of $\x\jjj$ by reflectedly periodizing, as the support of $\x\jjj$ has to be contained in $I_{2^j-\widetilde m\jj,2^j+\widetilde m\jj-1}$. Analogously, $\z\jjj_{(0)}$ and $\z\jjj_{(1)}$ take into account the at most $\widetilde m\jj$ nonzero entries of $\x\jjj_{(0)}$ and $\x\jjj_{(1)}$. Note that the restrictions still satisfy
  \begin{equation}\label{eq:per_z}
   \z\jj=\z\jjj_{(0)}+\J_{2^{\tilde K-1}}\z\jjj_{(1)}.
  \end{equation}
  Therefore, it is enough to derive a fast algorithm for computing $\z\jjj_{(0)}$, using $\z\jj$ and $2^{\widetilde K}$ entries of $\left(\x\jjj\right)\dct$. Recall that it follows from (\ref{eq:even_odd}) that
  \begin{equation}\label{eq:even_odd2}
   \left(\left(x\jjj\right)\dct_{2k+1}\right)_{k=0}^{2^j-1}=\frac{1}{\sqrt{2}}\Cf_{2^j}\left(2\x\jjj_{(0)}-\x\jj\right).
  \end{equation}
  We can restrict (\ref{eq:even_odd2}) to the vectors $\z\jj$ and $\z\jjj_{(0)}$, which yields
  \begin{align}
   &\left(\left(x\jjj\right)\dct_{2k+1}\right)_{k=0}^{2^j-1} \notag \\
   =&\frac{1}{\sqrt{2^j}}\left(\cos\left(\frac{(2k+1)(2l'+1)\pi}{4\cdot 2^j}\right)\right)_{k=0,\,l'=2^j-2^{\tilde K-1}}^{2^j-1}\left(2\z\jjj_{(0)}-\z\jj\right) \notag \\
   =&\frac{1}{\sqrt{2^j}}\left(\cos\left(\frac{(2k+1)(2^{j+1}-(2l+1))\pi}{4\cdot 2^j}\right)\right)_{k,\,l=0}^{2^j-1,\,2^{\tilde K-1}-1} \J_{2^{\tilde K-1}}\left(2\z\jjj_{(0)}-\z\jj\right) \notag \\
   =&\frac{1}{\sqrt{2^j}}\left(\cos\left(\frac{(2k+1)2^{j+1}\pi}{4\cdot 2^j}\right)\cos\left(\frac{(2k+1)(2l+1)\pi}{4\cdot2^j}\right)\right. \notag \\
   &+\left.\sin\left(\frac{(2k+1)2^{j+1}\pi}{4\cdot2^j}\right)\sin\left(\frac{(2k+1)(2l+1)\pi}{4\cdot2^j}\right)\right)_{k,\,l=0}^{2^j-1,\,2^{\tilde K-1}-1} \J_{2^{\tilde K-1}} \left(2\z\jjj_{(0)}-\z\jj\right) \notag \\
   =&\frac{1}{\sqrt{2^j}}\left((-1)^k\sin\left(\frac{(2k+1)(2l+1)\pi}{4\cdot 2^j}\right)\right)_{k,\,l=0}^{2^j-1,\,2^{\tilde K-1}-1}\J_{2^{\tilde K-1}}\left(2\z\jjj_{(0)}-\z\jj\right), \label{eq:restrict_this}
  \end{align}
  where $l\coloneqq2^j-1-l'$. As $\z\jj$ and $\z_{(0)}\jjj$ have length $2^{\tilde K-1}$, it suffices to consider the $2^{\tilde K-1}$ equations corresponding to the indices $2k_p+1$, where $k_p\coloneqq2^{j-\tilde K}(2p+1)$, $p\in\left\{0,\dotsc,2^{\tilde K-1}-1\right\}$. We obtain
  \begin{align}
   &\sqrt{2^j}(-1)^{2^{j-\tilde K}}\left(\left(x\jjj\right)\dct_{2k_p+1}\right)_{p=0}^{2^{\tilde K-1}-1} \notag \\
   =&\left(\sin\left(\frac{\left(2^{j-\tilde K+1}(2p+1)+1\right)(2l+1)\pi}{4\cdot2^j}\right)\right)_{p,\,l=0}^{2^{\tilde K-1}-1} \J_{2^{\tilde K-1}}\left(2\z\jjj_{(0)}-\z\jj\right) \notag \\
   =&\left(\sin\left(\frac{(2p+1)(2l+1)\pi}{4\cdot2^{\tilde K-1}}\right)\cos\left(\frac{(2l+1)\pi}{4\cdot2^j}\right)\right. \notag \\
   &+\left.\cos\left(\frac{(2p+1)(2l+1)\pi}{4\cdot2^{\tilde K-1}}\right)\sin\left(\frac{(2l+1)\pi}{4\cdot2^j}\right)\right)_{p,\,l=0}^{2^{\tilde K-1}-1}\J_{2^{\tilde K-1}}\left(2\z\jjj_{(0)}-\z\jj\right).\label{eq:some_step}
  \end{align}
  Defining the vectors
  \[
   \cc\coloneqq\left(\cos\left(\frac{(2l+1)\pi}{4\cdot2^j}\right)\right)_{l=0}^{2^{\tilde K-1}-1} \quad\text{and}\quad 
   \s\coloneqq\left(\sin\left(\frac{(2l+1)\pi}{4\cdot2^j}\right)\right)_{l=0}^{2^{\tilde K-1}-1},
  \]
  (\ref{eq:some_step}) can be written as
  \begin{align}
   &\sqrt{2^{j-\tilde K+2}}(-1)^{2^{j-\tilde K}}\left(\left(x\jjj\right)\dct_{2k_p+1}\right)_{p=0}^{2^{\tilde K-1}-1} \notag \\
   =&\left(\Sf_{2^{\tilde K-1}}\cdot\diag(\cc)+\Cf_{2^{\tilde K-1}}\cdot\diag(\s)\right) \J_{2^{\tilde K-1}}\left(2\z\jjj_{(0)}-\z\jj\right) \notag \\
   =&\left(\Cf_{2^{\tilde K-1}}\diag(\s)+\J_{2^{\tilde K-1}}\Cf_{2^{\tilde K-1}}\D_{2^{\tilde K-1}}\diag(\cc)\right) \J_{2^{\tilde K-1}}\left(2\z\jjj_{(0)}-\z\jj\right) \notag \\
   =&\left(\begin{array}{c|c}
                                                   \Cf_{2^{\tilde K-1}} & \J_{2^{\tilde K-1}}\Cf_{2^{\tilde K-1}}
                                                  \end{array}\right)
                                                  \begin{pmatrix}
                                                   \diag(\s) & \\
                                                   & \D_{2^{\tilde K-1}}\diag(\cc)
                                                  \end{pmatrix} 
                                                  \begin{pmatrix}
                                                   \J_{2^{\tilde K-1}}\left(2\z\jjj_{(0)}-\z\jj\right) \\
                                                   \J_{2^{\tilde K-1}}\left(2\z\jjj_{(0)}-\z\jj\right)
                                                  \end{pmatrix} \label{eq:first_half}
  \end{align}
  where we used a connection between the sine and cosine matrices of type IV (see \cite{wang}),
  \[
   \Sf_n=\J_n\Cf_n\D_n \quad\text{with}\quad \D_n=\diag\left((-1)^k\right)_{k=0}^{n-1} \quad \forall\, n\in\NN.
  \]
  As the first matrix in (\ref{eq:first_half}) is not a square matrix, we consider $2^{\tilde K-1}$ additional equations from (\ref{eq:restrict_this}). Now we choose the equations corresponding to the indices $2k'_p+1$, where $k_p'\coloneqq2^{j-\tilde K}(2p+1)-1$, $p\in\left\{0,\dotsc,2^{\tilde K-1}-1\right\}$.  Then we find that
  \begin{align}
   &\sqrt{2^{j-\tilde K+2}}\left(\left(x\jjj\right)\dct_{2k'_p+1}\right)_{p=0}^{2^{\tilde K-1}-1} \notag \\
   =&\frac{(-1)^{k'_p}}{\sqrt{2^{\tilde K-2}}}\left(\sin\left(\frac{\left(2^{j-\tilde K+1}(2p+1)-1\right)(2l+1)\pi}{4\cdot2^j}\right)\right)_{p,\,l=0}^{2^{\tilde K-1}-1} \J_{2^{\tilde K-1}}\left(2\z\jjj_{(0)}-\z\jj\right) \notag \\
   =&\frac{(-1)^{2^{j-\tilde K}-1}}{\sqrt{2^{\tilde K-2}}}\left(\sin\left(\frac{(2p+1)(2l+1)\pi}{4\cdot2^{\tilde K-1}}\right)\cos\left(\frac{(2l+1)\pi}{4\cdot2^j}\right)\right. \notag \\
   &-\left.\cos\left(\frac{(2p+1)(2l+1)\pi}{4\cdot2^{\tilde K-1}}\right)\sin\left(\frac{(2l+1)\pi}{4\cdot2^j}\right)\right)_{p,\,l=0}^{2^{\tilde K-1}-1}\J_{2^{\tilde K-1}}\left(2\z\jjj_{(0)}-\z\jj\right) \notag \\
   =&\frac{(-1)^{2^{j-\tilde K}-1}}{\sqrt{2^{\tilde K-2}}}\sqrt{2^{\tilde K-2}}\left(\Sf_{2^{\tilde K-1}}\cdot\diag(\cc)-\Cf_{2^{\tilde K-1}}\cdot\diag(\s)\right)\J_{2^{\tilde K-1}}\left(2\z\jjj_{(0)}-\z\jj\right) \notag \\
   =&(-1)^{2^{j-\tilde K}}\left(\begin{array}{c|c}
                                                   \Cf_{2^{\tilde K-1}} & -\J_{2^{\tilde K-1}}\Cf_{2^{\tilde K-1}}
                                                  \end{array}\right)
                                                  \begin{pmatrix}
                                                   \diag(\s) & \\
                                                   & \D_{2^{\tilde K-1}}\diag(\cc)
                                                  \end{pmatrix} \notag \\
                                                  &\cdot\begin{pmatrix}
                                                   \J_{2^{\tilde K-1}}\left(2\z\jjj_{(0)}-\z\jj\right) \\
                                                   \J_{2^{\tilde K-1}}\left(2\z\jjj_{(0)}-\z\jj\right)
                                                  \end{pmatrix} \label{eq:second_half}
  \end{align}
  Using Lemma \ref{lem:periodization} we denote by 
  \begin{align*}
    &\bb^0\coloneqq\left(\left(x\jjj\right)\dct_{2k_p+1}\right)_{p=0}^{2^{\tilde K-1}-1} =\sqrt{2}^{J-j-1}\left(\cx_{2^{J-j-1}(2k_p+1)}\right)_{p=0}^{2^{\tilde K-1}-1}\in\RR^{2^{\tilde K-1}}
   \quad\text{and} \\
    &\bb^1\coloneqq\left(\left(x\jjj\right)\dct_{2k'_p+1}\right)_{p=0}^{2^{\tilde K-1}-1} =\sqrt{2}^{J-j-1}\left(\cx_{2^{J-j-1}(2k'_p+1)}\right)_{p=0}^{2^{\tilde K-1}-1}\in\RR^{2^{\tilde K-1}}
  \end{align*}
  the vectors of required entries of $\cxx$. Combining (\ref{eq:first_half}) and (\ref{eq:second_half}) yields
  \begin{align}
   &\sqrt{2^{j-\tilde K+2}}(-1)^{2^{j-\tilde K}}\begin{pmatrix}
    \bb^0 \\
    \bb^1 
   \end{pmatrix} \notag \\
   =& \begin{pmatrix}
       \Cf_{2^{\tilde K-1}} & \J_{2^{\tilde K-1}}\Cf_{2^{\tilde K-1}} \\
       \Cf_{2^{\tilde K-1}} & -\J_{2^{\tilde K-1}}\Cf_{2^{\tilde K-1}}
      \end{pmatrix}
      \begin{pmatrix}
       \diag(\s) & \\
       & \D_{2^{\tilde K-1}}\diag(\cc)
      \end{pmatrix} 
      \begin{pmatrix}
       \J_{2^{\tilde K-1}}\left(2\z\jjj_{(0)}-\z\jj\right) \\
       \J_{2^{\tilde K-1}}\left(2\z\jjj_{(0)}-\z\jj\right)
      \end{pmatrix} \notag \\
   = &\begin{pmatrix}
      \I_{2^{\tilde K-1}} & \J_{2^{\tilde K-1}} \\
      \I_{2^{\tilde K-1}} & -\J_{2^{\tilde K-1}}
     \end{pmatrix}
     \begin{pmatrix}
      \Cf_{2^{\tilde K-1}} & \\
      & \Cf_{2^{\tilde K-1}}
     \end{pmatrix}
     \begin{pmatrix}
      \diag(\s) & \\
      & \D_{2^{\tilde K-1}}\diag(\cc)
     \end{pmatrix} \notag \\
     &\cdot\begin{pmatrix}
      \J_{2^{\tilde K-1}}\left(2\z\jjj_{(0)}-\z\jj\right) \\
      \J_{2^{\tilde K-1}}\left(2\z\jjj_{(0)}-\z\jj\right)
     \end{pmatrix}. \label{eq:all}
  \end{align}
  Note that the first matrix in (\ref{eq:all}) is invertible, as
  \[
   \begin{pmatrix}
      \I_{2^{\tilde K-1}} & \J_{2^{\tilde K-1}} \\
      \I_{2^{\tilde K-1}} & -\J_{2^{\tilde K-1}}
     \end{pmatrix} \cdot\frac{1}{2}
   \begin{pmatrix}
      \I_{2^{\tilde K-1}} & \I_{2^{\tilde K-1}} \\
      \J_{2^{\tilde K-1}} & -\J_{2^{\tilde K-1}}
     \end{pmatrix}
   =\begin{pmatrix}
     \I_{2^{\tilde K-1}} & \\
     & \I_{2^{\tilde K-1}}
    \end{pmatrix}.
  \]
  Furthermore, since $\widetilde m\jj\leq M$ and thus $\tilde K\leq L\leq j$,
  \[
   \frac{(2l+1)\pi}{4\cdot2^j}\in\left(0,\frac{\pi}{4}\right)
  \]
  for all $l\in\left\{0,\dotsc,2^{\tilde K-1}-1\right\}$. Consequently, we have that
  \[
   \cos\left(\frac{(2l+1)\pi}{4\cdot2^j}\right)\in\left(\frac{1}{\sqrt{2}},1\right) \qquad\text{and}\qquad 
   \sin\left(\frac{(2l+1)\pi}{4\cdot2^j}\right)\in\left(0,\frac{1}{\sqrt{2}}\right),
  \]
  which means that the third matrix in (\ref{eq:all}) is invertible as well, since the multiplication of the second half of the odd diagonal entries with $-1$, caused by $\D_{2^{\tilde K-1}}$, does not change the absolute value of the determinant of the matrix. Thus all matrices in (\ref{eq:all}) are invertible and it follows that
  \begin{align}
   &\begin{pmatrix}
    \J_{2^{\tilde K-1}}\left(2\z\jjj_{(0)}-\z\jj\right) \\
    \J_{2^{\tilde K-1}}\left(2\z\jjj_{(0)}-\z\jj\right)
   \end{pmatrix} \notag \\
   =&\sqrt{2^{j-\tilde K}}(-1)^{2^{j-\tilde K}}\begin{pmatrix}
                                   \diag(\tilde\s) & \\
                                   & \diag(\tilde\cc)\D_{2^{\tilde K-1}}
                                  \end{pmatrix}
                                  \begin{pmatrix}
                                   \Cf_{2^{\tilde K-1}} & \\
                                   & \Cf_{2^{\tilde K-1}}
                                  \end{pmatrix} \notag \\
                                  &\cdot\begin{pmatrix}
                                   \I_{2^{\tilde K-1}} & \I_{2^{\tilde K-1}} \\
                                   \J_{2^{\tilde K-1}} & -\J_{2^{\tilde K-1}}
                                  \end{pmatrix}
                                  \begin{pmatrix}
                                   \bb^0 \\
                                   \bb^1
                                  \end{pmatrix} \notag \\
   =&\sqrt{2^{j-\tilde K}}(-1)^{2^{j-\tilde K}}\begin{pmatrix}
                                  \diag(\tilde\s)\Cf_{2^{\tilde K-1}} & \\
                                  & \diag(\tilde\cc)\D_{2^{\tilde K-1}}\Cf_{2^{\tilde K-1}} 
                                 \end{pmatrix}
                                 \begin{pmatrix}
                                  \bb^0+\bb^1 \\
                                  \J_{2^{\tilde K-1}}\left(\bb^0-\bb^1\right)
                                 \end{pmatrix} \label{eq:almost}
  \end{align}
  where
  \[
   \tilde\cc\coloneqq\left(\cos\left(\frac{(2l+1)\pi}{4\cdot2^j}\right)^{-1}\right)_{l=0}^{2^{\tilde K-1}-1} \qquad\text{and}\qquad 
   \tilde\s\coloneqq\left(\sin\left(\frac{(2l+1)\pi}{4\cdot2^j}\right)^{-1}\right)_{l=0}^{2^{\tilde K-1}-1}.
  \]
  Using only the second $2^{\tilde K-1}$ equations in (\ref{eq:almost}) we obtain
  \[
   \z\jjj_{(0)}=\frac{1}{2}\left(\sqrt{2^{j-\tilde K}}(-1)^{2^{j-\tilde K}}\J_{2^{\tilde K-1}}\diag(\tilde\cc)\D_{2^{\tilde K-1}}\Cf_{2^{\tilde K-1}}\J_{2^{\tilde K-1}}\left(\bb^0-\bb^1\right)+\z\jj\right),
  \]
  which implies that $\z\jjj_{(0)}$ can be computed in $\mathcal{O}\left(2^{\tilde K-1}\log2^{\tilde K-1}\right)$ operations using $2^{\tilde K}$ entries of $\cxx$, as $\D_{2^{\tilde K-1}}$ is a diagonal matrix and $\J_{2^{\tilde K-1}}$ is a permutation. Then $\z\jjj_{(1)}$ can be found in $\mathcal{O}\left(2^{\tilde K-1}\right)$ time by (\ref{eq:per_z}) and $\x\jjj$ is given as
  \[
   x\jjj_k=\begin{cases}
            \left(z\jjj_{(0)}\right)_{k-2^j+2^{\tilde K-1}} & \text{if } k\in\left\{2^j-2^{\tilde K-1},\dotsc,2^j-1\right\}, \\
            \left(z\jjj_{(1)}\right)_{k-2^j} & \text{if } k\in\left\{2^j,\dotsc,2^j+2^{\tilde K-1}-1\right\}, \\
            0 & \text{else,}
           \end{cases}
  \]
  since all possibly nonzero entries of $\x\jjj$ are determined by $\z\jjj_{(0)}$ and $\z\jjj_{(1)}$.
 \end{proof}
 Note that by choosing the second $2^{\tilde K-1}$ equations in (\ref{eq:almost}) we avoid inverting $\diag(\s)$, which would be numerically less stable, since for large $\tilde K$ its nonzero entries are rather close to zero, whereas all nonzero entries of $\diag(\cc)$ are greater than $\frac{1}{\sqrt{2}}$.

\section{The Sparse DCT Algorithms}\label{sec:algs}
In Section \ref{sec:procedures} we introduced all procedures necessary for the new sparse DCT for vectors $\x\in\RR^{2^J}$ with short support of length $m\leq M$ that satisfy (\ref{eq:suppose}).
\subsection{The Sparse DCT for Bounded Short Support Length}\label{sec:alg_bound}
We suppose that $N=2^J$ and $\x\in\RR^N$ has a short support of unknown length $m$, but that a bound $M\geq m$ is known. Further, we assume that (\ref{eq:suppose}) holds for $\x$ and that we can access all entries of $\cxx\in\RR^N$. The algorithm begins by computing the initial vector 
\[
 \x^{(L)}=\mathbf{C}^{\mathrm{III}}_{2^L}\left(\sqrt{2}^{J-L}\left(\cx_{2^{J-L}k}\right)_{k=0}^{2^L-1}\right),
\]
where $L\coloneqq\left\lceil\log_2M\right\rceil+1$, using a fast DCT-III algorithm for vectors with full support, see, e.g., \cite{plonka_dct,wang}, since DCT-III is the same as IDCT-II. For $j\in\{L,\dotsc,J-1\}$ we perform the following iteration steps.
\begin{enumerate}[label=\arabic*),ref=\arabic*]
 \item If the support of $\x\jj$ is not contained in $I_{2^j-M,2^j-1}$, recover $\x\jjj$ using the DCT procedure given in Theorem \ref{thm:nocoll}.
 \item \label{item:step_coll} If the support of $\x\jj$ is contained in $I_{2^j-M,2^j-1}$, recover $\x\jjj$ using the DCT procedure given in Theorem \ref{thm:coll}.
\end{enumerate}
It follows from Lemma \ref{lem:facts} that there is at most one index $j'$ s.t.\ $S^{\left(j'\right)}\subset I_{2^{j'}-M,2^{j'}-1}$. Hence, we have to apply step \ref{item:step_coll} at most once. The complete procedure is summarized in Algorithm \ref{alg:dct_bound}.
\begin{algorithm}
\caption{Sparse Fast DCT for Vectors with Bounded Short Support Length}
\label{alg:dct_bound}
\begin{algorithmic}[1]
\small
\renewcommand{\algorithmicrequire}{\textbf{Input:}}
\renewcommand{\algorithmicensure}{\textbf{Output:}}
\Require $\cxx$, where $\x\in\RR^N$, $N=2^J$, $J\in\NN$, has an unknown short support of length at most $M$ and satisfies (\ref{eq:suppose}), $M$ and noise threshold $\eps>0$.
\State $L\gets\lceil\log_2M\rceil+1$ and $\x^{(L)}\gets\mathbf{DCT\text{-}III}\left[\sqrt{2}^{J-L}\left(\cx_{2^{J-L}k}\right)_{k=0}^{2^L-1}\right]$ \label{line:initialize}
\State Find $\mu^{(L)}$ and $m^{(L)}$. \label{line:initial_supp}
\For{$j$ from $L$ to $J-1$} 
\If{$\mu\jj<2^j-M$}
\State Find $\alpha=\sqrt{2}^{J-j-1}\cx_{2^{J-j-1}(2k_0+1)}\neq0$. \label{line:nonzero}
\State $\left(u^0\right)\dct_{2k_0+1}\gets\frac{1}{\sqrt{2^j}}\sum\limits_{l=0}^{m\jj-1}\cos\left(\frac{(2k_0+1)\left(2\left(\mu\jj+l\right)+1\right)\pi}{2\cdot2^{j+1}}\right)x\jj_{\mu\jj+l}$ \label{line:nocoll_start}
\State $\nu_t\gets\begin{cases}
                   0 & \text{if } \left|\left(u^0\right)\dct_{2k_0+1}-\alpha\right| 
                       <\left|\left(u^0\right)\dct_{2k_0+1}+\alpha\right|, \\
                   1 & \text{else}
                  \end{cases}$
\If{$\nu_t=0$}
\State $\mu\jjj\gets\mu\jj$ and $m\jjj\gets m\jj$ \label{line:mu_m_0}
\State $x\jjj_k\gets\begin{cases} 
                     x\jj_k & \text{if } k\in\left\{\mu\jjj,\dotsc,\mu\jjj+m\jjj-1\right\} \\
                     0 & \text{else}
                    \end{cases}$ 
\Else 
\State $\mu\jjj\gets 2^{j+1}-m\jj-\mu\jj$ and $m\jjj\gets m\jj$ \label{line:mu_m_1}
\State $x\jjj_k\gets\begin{cases}
                     x\jj_{2^{j+1}-1-k} & \text{if } k\in\left\{\mu\jjj,\dotsc,\mu\jjj+m\jjj-1\right\} \\
                     0 & \text{else}
                    \end{cases}$ \label{line:nocoll_end}
\EndIf
\Else 
\State $\tilde K\gets \left\lceil\log_2\left(2^j-\mu\jj\right)\right\rceil+1$
\State $\z\jj\gets\left(x\jj_{2^j-2^{\tilde K-1}+k}\right)_{k=0}^{2^{\tilde K-1}-1}$ \label{line:coll_start}
\State $\bb^0\gets\sqrt{2}^{J-j-1}\left(\cx_{2^{J-\tilde K}(2p+1)+2^{J-j-1}}\right)_{p=0}^{2^{\tilde K-1}-1}$ \label{line:samples_nocoll_1}
\State $\bb^1\gets\sqrt{2}^{J-j-1}\left(\cx_{2^{J-\tilde K}(2p+1)-2^{J-j-1}}\right)_{p=0}^{2^{\tilde K-1}-1}$ \label{line:samples_nocoll_2}
\State \begin{varwidth}[t]{\linewidth} $\z\jjj_{(0)}\gets\frac{1}{2}\sqrt{2^{j-\tilde K}}(-1)^{2^{j-\tilde K}}\J_{2^{\tilde K-1}}\diag(\tilde\cc)\D_{2^{\tilde K-1}}\mathbf{DCT\text{-}IV}\left[\J_{2^{\tilde K-1}}\left(\bb^0-\bb^1\right)\right]$ \\
\hskip\algorithmicindent\phantom{$\z\jjj_{(0)}\gets$} $+\frac{1}{2}\z\jj$ \label{line:z_0} \end{varwidth}
\State $\left(z\jjj_{(0)}\right)_k\gets\begin{cases}
                     \left(z\jjj_{(0)}\right)_k & \text{if } \left(z\jjj_{(0)}\right)_k>\eps, \\
                     0 & \text{else,}
                    \end{cases} \qquad k\in\left\{0,\dotsc,2^{\tilde K-1}-1\right\}$ \label{line:z_0_2}
\State $\z\jjj_{(1)}\gets\J_{2^{\tilde K-1}}\left(\z\jj-\z\jjj_{(0)}\right)$ \label{line:z_1}
\State $x\jjj_k\gets\begin{cases}
                     \left(z\jjj_{(0)}\right)_{k-2^j+2^{\tilde K-1}} & \text{if } k\in\left\{2^j-2^{\tilde K-1},\dotsc,2^j-1\right\} \\
                     \left(z\jjj_{(1)}\right)_{k-2^j} & \text{if } k\in\left\{2^j,\dotsc,2^j+2^{\tilde K-1}-1\right\} \\
                     0 & \text{else}
                    \end{cases}$ \label{line:x_coll}
\State Find $\mu\jjj$ and $m\jjj$. \label{line:find_supp}
\EndIf
\EndFor
\Ensure $\x=\x^{(J)}$
\end{algorithmic}
\end{algorithm}

\begin{rem}
 For finding the first support index $\mu^{(L)}$ and the support length $m^{(L)}$ in line \ref{line:initial_supp}, as well as $\mu\jjj$ and $m\jjj$ in line \ref{line:find_supp} efficiently, we choose a threshold $\eps>0$ depending on the noise level of the data. If we want to determine the support of $\x^{(L)}$, we define the set
 \[
  T^{(L)}\coloneqq\left\{k\in I_{0,2^L-1}:x^{(L)}_k>\eps\right\}\eqqcolon\left\{u_1,\dotsc,u_P\right\}
 \]
 of indices corresponding to significantly large entries of $\x^{(L)}$. This set can be found in $\mathcal{O}\left(2^L\right)=\mathcal{O}\left(M\right)$ time, and we set
 \[
  \mu^{(L)}\coloneqq u_1 \quad\text{and}\quad m^{(L)}\coloneqq u_P-u_1+1.
 \]
 For $j\in\{L,\dotsc,J-1\}\backslash\{j'\}$, i.e., if $\x\jjj$ is computed with the DCT proedure given in Theorem \ref{thm:nocoll}, $\mu\jjj$ and $m\jjj$ are computed in line \ref{line:mu_m_0} or \ref{line:mu_m_1}. In order to find the support of $\x\jjj$ for $j=j'\in\{L,\dotsc,J-1\}$, i.e., if $\x\jjj$ is obtained by the DCT procedure given in Theorem \ref{thm:coll}, it suffices to consider the set 
 \[
  T\jjj\coloneqq\left\{k\in\left\{2^j-2^{\tilde K-1},\dotsc,2^j+2^{\tilde K-1}-1\right\}:x\jjj_k>\eps\right\}\eqqcolon\left\{v_1,\dotsc,v_Q\right\},
 \]
 where $Q\leq\widetilde m\jj\leq M$. Then $T\jjj$ can be found in $\mathcal{O}\left(2^{\tilde K}\right)=\mathcal{O}\left(\widetilde m\jj\right)=\mathcal{O}(M)$ time as well, and we define
 \[
  \mu\jjj\coloneqq v_1 \quad\text{and}\quad m\jjj\coloneqq v_Q-v_1+1.
 \]
\end{rem}
Having presented our new algorithm we now prove that its runtime and sampling complexity are sublinear in the vector length $N$.
\begin{thm}\label{thm:runtime_bound}
 Let $\x\in\RR^N$, $N=2^J$, $J\in\NN$, have a short support of length $m$ and assume that $\x$ satisfies $(\ref{eq:suppose})$. Further suppose that only an upper bound $M\geq m$ is known. Then Algorithm $\ref{alg:dct_bound}$ has a runtime of $\mathcal{O}\left(M\log M+m\log_2\frac{N}{M}\right)$ and uses $\mathcal{O}\left(M+m\log_2\frac{N}{M}\right)$ samples of $\cxx$.
\end{thm}
\begin{proof}
 Computing the initial vector $\x^{(L)}$ in line \ref{line:initialize} via a $2^L$-length DCT-III has a runtime of $\mathcal{O}\left(2^L\log2^L\right)$, see, e.g., \cite{plonka_dct,wang}, and finding $\mu^{(L)}$ and $m^{(L)}$ needs $\mathcal{O}\left(2^L\right)$ operations.
 
 For $j\in\{L,\dotsc,J-1\}\backslash\{j'\}$ the support of $\x\jj$ is not contained in $I_{2^j-M,2^j-1}$; hence, we have to apply the procedure from Theorem \ref{thm:nocoll}. Finding a nonzero entry in line \ref{line:nonzero} requires $\mathcal{O}\left(m\jj\right)=\mathcal{O}(m)$ operations by Lemma \ref{lem:nonzero} and executing lines \ref{line:nocoll_start} to \ref{line:nocoll_end} has a runtime of $\mathcal{O}\left(m\jj\right)$ as well.
 
 If $j=j'$, we use the method from Theorem \ref{thm:coll}. The computation of $\z\jjj_{(0)}$ in lines \ref{line:z_0} and \ref{line:z_0_2} requires a DCT-IV of length $2^{\tilde K-1}$ and further operations of complexity $\mathcal{O}\left(2^{\tilde K-1}\right)$, since $\D_{2^{\tilde K-1}}$ and $\diag(\tilde\cc)$ are diagonal and $J_{2^{\tilde K-1}}$ is a permutation. Computing $\z\jjj_{(1)}$ and $\x\jjj$ in lines \ref{line:z_1} and \ref{line:x_coll} and finding $\mu\jjj$ and $m\jjj$ in line \ref{line:find_supp} needs $\mathcal{O}\left(2^{\tilde K-1}\right)$ operations. Note that we can only estimate that $\widetilde m\jj=\mathcal{O}(M)$ and thus $2^{\tilde K-1}=\mathcal{O}(M)$, since $m$ is not known apriori and the support of $\x\jj$ can be located anywhere in $I_{2^j-M,2^j-1}$. Thus, lines \ref{line:coll_start} to \ref{line:find_supp} have a runtime of $\mathcal{O}\left(2^{\tilde K-1}\log2^{\tilde K-1}\right)=\mathcal{O}(M\log M)$. 
 
 Consequently, Algorithm \ref{alg:dct_bound} has an overall runtime of
 \[
  \mathcal{O}\left(\sum_{\substack{j=L \\ j\neq j'}}^{J-1}m\jj+2^{\tilde K}\log 2^{\tilde K}\right)
  =\mathcal{O}\left((J-L)m+M\log M\right)=\mathcal{O}\left(M\log M+m\log_2\frac{N}{M}\right).
 \]
 The initial vector $\x^{(L)}$ can be computed from $2^L$ samples of $\cxx$ in line \ref{line:initialize}. Finding an oddly indexed nonzero entry of $\left(\x\jjj\right)\dct$ in line \ref{line:nonzero} requires at most $m\jj$ samples of $\cxx$ by Lemma \ref{lem:nonzero}. Further, we need to take $2^{\tilde K}\leq 2^L$ samples in lines \ref{line:samples_nocoll_1} and \ref{line:samples_nocoll_2}, which yields a total sampling complexity of
 \[
  \mathcal{O}\left(\sum_{\substack{j=L \\ j\neq j'}}^{J-1}m\jj+2^L\right)
  =\mathcal{O}\left((J-L)m+M\right)=\mathcal{O}\left(M+m\log_2\frac{N}{M}\right).
 \]
\end{proof}

\subsection{The Sparse DCT for Exactly Known Short Support Length}\label{sec:alg_exact}
Having introduced our new sparse DCT for vectors with bounded short support length we can now modify Algorithm \ref{alg:dct_bound} to better fit the case where the support length $m$ of $\x$ is known exactly, i.e., if $M=m$. Since there is at most one index $j'$ for which the support of $\x^{(j')}$ is contained in the last $m$ entries, the procedure from Theorem \ref{thm:coll} only has to be applied if $m^{(j')}<m^{(j'+1)}$, i.e., if there was a collision of nonzero entries, or if $\nu^{(j')}=2^{j'}-1$, unlike in Algorithm \ref{alg:dct_bound}. 

We can simply replace $M$ by $m$ and $L=\left\lceil\log_2 M\right\rceil+1$ by $L\coloneqq\left\lceil\log_2 m\right\rceil+1$ in Algorithm \ref{alg:dct_bound} to obtain the sparse DCT for vectors with exactly known short support length. Then $\widetilde m^{(j')}=2^{j'}-\mu^{(j')}=m^{(j')}=\mathcal{O}(m)$ and $\tilde K=L$.  Note that $m\jjj=m$ for $j\geq j'$. We find the following runtime and sampling complexities.
\begin{thm}\label{thm:runtime_exact}
 Let $\x\in\RR^N$, $N=2^J$, $J\in\NN$, have a short support of length $m$ and assume that $\x$ satisfies $(\ref{eq:suppose})$. Further suppose that $m$ is known exactly. Then Algorithm $\ref{alg:dct_bound}$ has a runtime of $\mathcal{O}\left(m\log m+m\log_2\frac{N}{m}\right)$ and uses $\mathcal{O}\left(m+m\log_2\frac{N}{m}\right)$ samples of $\cxx$.
\end{thm}
\section{Numerics}\label{sec:numerics}
In the following section we evaluate the performance of the variant of Algorithm \ref{alg:dct_bound} for exactly known support lengths and the variant for bounded short support lengths with respect to runtime and robustness to noise. To the best of our knowledge most existing sparse DCT algorithm use a the approach of computing $\x$ by recovering $\y=(\x^T,(\J_N\x)^T)^T$ from $\widehat\y$ by an unstructured and thus inefficient $2m$-sparse IFFT. Only Algorithm 2 in \cite{bit_plon} uses an IFFT especially tailored to the structure of $\y$, so we only compare the variants of our algorithm to this method and to \textsc{Matlab} 2018a's \texttt{idct} routine, which is part of the \emph{Signal Processing Toolbox}, see \cite{matlab_idct}. \texttt{idct} is a fast and highly optimized implementation of the fast inverse cosine transform of type II. Note that, compared to the implementation of \texttt{idct} in \textsc{Matlab} 2016b, which we used for the numerical experiments in \cite{bit_plon}, the runtime of \texttt{idct} in \textsc{Matlab} 2018a has reduced by almost half for arbitrary nonnegative vectors of length $N=2^{20}$ on the machine used for the experiments, which is why the results of the numerical experiments with respect to runtime in this section are different from the ones in \cite{bit_plon}, Section 6.2. All algorithms have been implemented in \textsc{Matlab} 2018a, and the code is freely available in \cite{nam_code_dct,nam_code_dct_short}. Note that Algorithm 2 in \cite{bit_plon} does not require any a priori knowledge of the support length, but needs that for $\x\in\RR^{2^J}$ the vector $\y=(x_0,x_1,\dotsc,x_{N-1},x_{N-1},x_{N-2},\dotsc,x_0)^T\in\RR^{2^{J+1}}$ satisfies
\begin{equation}\label{eq:suppose_oneblock}
 \left|\sum_{l=0}^{2^{J+1-j}-1}y_{k+2^jl}\right|>\eps \qquad \forall\,j\in\{0,\dotsc,J+1\}
\end{equation}
for all $\left|y_k\right|>\eps$ for a noise threshold $\epsilon>0$. Algorithm \ref{alg:dct_bound}, on the other hand, requires an upper bound $M\geq m$ on the support length and that $\x\in\RR^{2^J}$ satisfies (\ref{eq:suppose}).

Figure \ref{fig:runtime} shows the average runtimes of Algorithm \ref{alg:dct_bound} for exactly known support lengths, i.e., for $M=m$, and for bounded short support lengths with $M=3m$, Algorithm 2 in \cite{bit_plon} and \texttt{idct} applied to $\cxx$ for 1{,}000 randomly generated $2^{20}$-length vectors $\x$ with short support of lengths varying between 10 and 500{,}000. 
\begin{figure}[!ht]
\centering
\resizebox{0.65\textwidth}{!}{\input{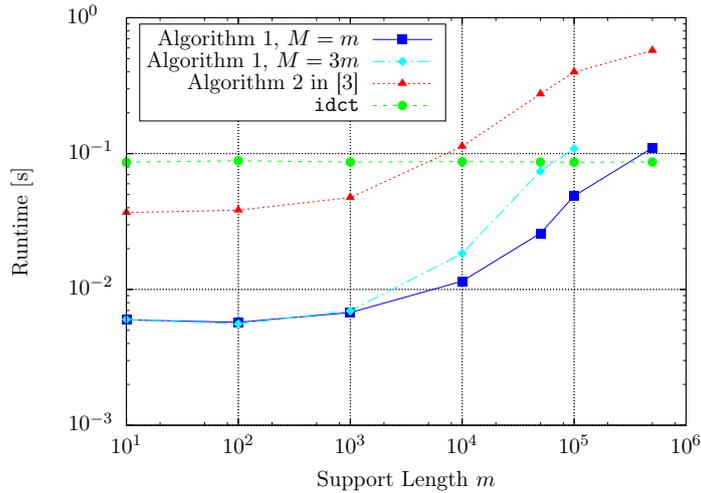}}
\caption{Average runtimes of Algorithm \ref{alg:dct_bound} for exactly known short support and for bounded short support and Algorithm 2 in \cite{bit_plon} with $\eps=10^{-4}$, and \textsc{Matlab}'s \texttt{idct} for 1{,}000 random input vectors with short support of length $m$, bound $M=3m$ and vector length $N=2^{20}$.} 
\label{fig:runtime}
\end{figure}
For Algorithm \ref{alg:dct_bound} and Algorithm 2 in \cite{bit_plon} we use the threshold $\eps=10^{-4}$. The nonzero entries of the vectors are chosen randomly with uniform distribution between 0 and 10, with $x_{\mu^{(J)}}$ and $x_{\nu^{(J)}}$ chosen from $(\eps,10]$. For each vector at most $\left\lfloor(m-2)/2\right\rfloor$ entries in the support block, excluding the first and last one, are randomly set to 0. Hence, both (\ref{eq:suppose}) and (\ref{eq:suppose_oneblock}) hold. Since for $m=500{,}000$ we have that $M=3m>N$, we only execute Algorithm \ref{alg:dct_bound} in the variant for bounded support lengths up to $m=100{,}000$. 

Of course the comparison of the sparse DCT algorithms to the highly optimized, support length independent \texttt{idct} routine must be flawed; however, one can see that all three sparse DCT procedures are much faster than \texttt{idct} for sufficiently small support lengths. For exactly known support lengths Algorithm \ref{alg:dct_bound} achieves smaller runtimes for block lengths up to $m=100{,}000$, for bounded support lengths this is the case for block lengths up to $m=50{,}000$, where the known bound on the block length is $M=150{,}000$, and for Algorithm 2 in \cite{bit_plon} for block lengths up to $m=1{,}000$. Note that by setting $\left\lfloor(m-2)/2\right\rfloor$ entries inside the support to zero, the actual sparsity of $\x$ can be almost as low as $m/2$; however, this does not affect the runtime of any of the considered algorithms. It follows from Table \ref{tab:error_ex}, presenting the average reconstruction errors for exact data for all four considered methods, that, while the sparse DCT algorithms do not achieve reconstruction errors comparable to those of \texttt{idct}, their outputs are still very accurate.
\begin{table}[!ht]
\begin{center}
\begin{tabular}{rcccc}\toprule
\multirow{2}{*}{$m$} & Algorithm \ref{alg:dct_bound}, & Algorithm \ref{alg:dct_bound},  & \multirow{2}{*}{Algorithm 2 in \cite{bit_plon}}	& \multirow{2}{*}{\texttt{idct}} \\
&  $M=m$ & $M=3m$ & &  \\ \midrule
10		& $1.8\cdot 10^{-20}$	& $1.7\cdot 10^{-20}$	& $1.3\cdot 10^{-19}$	& $7.8\cdot 10^{-21}$ \\
100		& $5.3\cdot 10^{-20}$	& $3.9\cdot 10^{-20}$	& $4.9\cdot 10^{-18}$	& $2.4\cdot 10^{-20}$ \\
1{,}000		& $7.5\cdot 10^{-14}$	& $4.1\cdot 10^{-14}$	& $4.9\cdot 10^{-13}$	& $7.6\cdot 10^{-20}$ \\
10{,}000	& $1.0\cdot 10^{-12}$	& $1.4\cdot 10^{-12}$	& $3.9\cdot 10^{-12}$	& $2.4\cdot 10^{-19}$ \\
50{,}000	& $3.6\cdot 10^{-12}$	& $2.9\cdot 10^{-12}$	& $1.5\cdot 10^{-11}$	& $5.4\cdot 10^{-19}$ \\
100{,}000 	& $7.5\cdot 10^{-12}$	& $7.6\cdot 10^{-19}$	& $2.9\cdot 10^{-11}$	& $7.6\cdot 10^{-19}$ \\
500{,}000	& $1.7\cdot 10^{-18}$	& $1.7\cdot 10^{-18}$	& $9.6\cdot 10^{-11}$	& $1.7\cdot 10^{-18}$ \\  \bottomrule 
\end{tabular}
\caption{Reconstruction errors for the four DCT algorithms for exact data.}
\label{tab:error_ex}
\end{center}
\end{table}

Further, we also investigate the robustness of Algorithm \ref{alg:dct_bound} for noisy data. We create disturbed cosine data $\z\dct\in\RR^N$ by adding uniform noise $\bfeta\in\RR^N$ to the given data $\cxx$,
\[
 \z\dct\coloneqq\cxx+\bfeta.
\]
We measure the noise with the \emph{signal-to-noise ratio (SNR)}, given by
\[
 \text{SNR}\coloneqq20\cdot\log_{10}\frac{\left\|\cxx\right\|_2}{\left\|\bfeta\right\|_2}.
\]
Figures \ref{fig:error_dct_100} and \ref{fig:error_dct_1000} depict the average reconstruction errors $\left\|\x-\x'\right\|_2/N$, where $\x$ denotes the original vector and $\x'$ the reconstruction by the corresponding algorithm applied to $\z\dct$ for support lengths $m=100$ and $m=1{,}000$. 
\begin{figure}[!tbp]
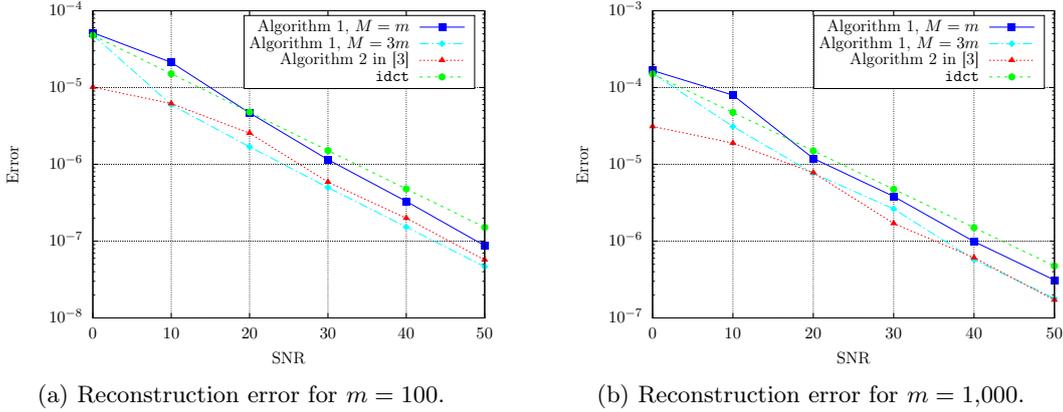

  \centering
  \subfloat[Reconstruction error for $m=100$.]{\resizebox{0.49\textwidth}{!}{\input{error_DCT_100_comp}}\label{fig:error_dct_100}}
  \hfill
  \subfloat[Reconstruction error for $m=1{,}000$.]{\resizebox{0.49\textwidth}{!}{\input{error_DCT_1000_comp}}\label{fig:error_dct_1000}}
  \caption{Average reconstruction errors $\|\x-\x'\|_2/N$ of Algorithm \ref{alg:dct_bound} for $M=m$ and $M=3m$, Algorithm 2 in \cite{bit_plon} and \texttt{idct} for 1{,}000 random input vectors with support length $m$ and vector length $N=2^{20}$.}
\end{figure}
The threshold parameters $\eps$ for both variants of Algorithm \ref{alg:dct_bound} and Algorithm 2 in \cite{bit_plon} are chosen according to Table \ref{tab:epsilon_dct}, where we use the $\eps$-values from \cite{bit_plon}, Section 6.2 for Algorithm 2 in said paper. All parameters were obtained in an attempt to minimize the reconstruction error and maximize the rate of correct recovery.
\begin{table}[!ht]
 \begin{center}
  \begin{tabular}{S[table-format=3.2]S[table-format=3.2]S[table-format=3.2]S[table-format=3.2]}\toprule
   {SNR} & {Alg. \ref{alg:dct_bound},} & {Alg. \ref{alg:dct_bound},} & {Alg. 2 in \cite{bit_plon}} \\
	& {$m=100$}	 & {$m=1{,}000$} & \\ \midrule
   0	& 2.50	& 2.50	& 2.50 \\
   10	& 2.00	& 2.10	& 1.80 \\
   20	& 1.00	& 1.50	& 1.00   \\
   30	& 0.40	& 0.85	& 0.50 \\
   40	& 0.15	& 0.20	& 0.15 \\
   50	& 0.05	& 0.10	& 0.05 \\ \bottomrule
  \end{tabular}
  \caption{Threshold $\varepsilon$ for Algorithm \ref{alg:dct_bound} and Algorithm 2 in \cite{bit_plon}.}
  \label{tab:epsilon_dct}
 \end{center}
\end{table}
For Algorithm \ref{alg:dct_bound} with $M=3m$ the reconstruction yields a smaller error than the one for \texttt{idct} and, for SNR values greater than 10, even a slightly smaller error than the one for Algorithm 2 in \cite{bit_plon} for $m=100$ and an error comparable to the one for Algorithm 2 in \cite{bit_plon} for $m=1{,}000$. For exactly known support lengths, the reconstruction yields a slightly smaller error than the one by \texttt{idct} for both support lengths. 

In certain applications it might be important to know the support of $\x$; hence, we also examine whether the sparse DCT algorithms can correctly identify the support for noisy input data. Tables \ref{tab:correct_recov_100} and \ref{tab:correct_recov_1000} show the rates of correct recovery of the support for $m=100$ and $m=1{,}000$. 
\begin{table}[!ht]
\begin{center}
\begin{tabular}{rS[table-format=3.2]S[table-format=3.2]S[table-format=3.2]S[table-format=3.2]S[table-format=3.2]}\toprule
& \multicolumn{5}{c}{Rate of Correct Recovery in \% for $m=100$} \\
\multirow{3}{*}{SNR} & {Alg. \ref{alg:dct_bound},} & {Alg. \ref{alg:dct_bound},} & {Alg. \ref{alg:dct_bound},} & {Alg. 2 in \cite{bit_plon}}	& {Alg. 2 in \cite{bit_plon},} \\
&  {$M=m$} & {$M=3m$} & {$M=3m$,} & &  \\ 
& & & {$m'\leq 3m$} & & {$m'\leq 3m$} \\ \midrule
0	& 61.6	& 89.9	& 0.0	& 83.1	& 77.1 \\
10	& 64.0	& 98.7	& 85.4	& 97.6	& 97.4 \\
20	& 95.1	& 100.0	& 96.2	& 100.0	& 100.0 \\
30	& 99.3	& 100.0	& 98.6	& 100.0	& 100.0 \\
40	& 99.9	& 100.0	& 99.4	& 100.0	& 100.0 \\
50	& 100.0	& 100.0	& 99.9	& 100.0	& 100.0 \\ \bottomrule
\end{tabular}
\caption{Rate of correct recovery of the support of $\x$ in \% for Algorithm \ref{alg:dct_bound} for $M=m$ and $M=3m$ and Algorithm 2 in \cite{bit_plon}, without bounding $m'$ and with $m'\leq 3m$, for 1{,}000 random input vectors with support length $m=100$ from Figure \ref{fig:error_dct_100}.}
\label{tab:correct_recov_100}
\end{center}
\end{table}
\begin{table}[!ht]
\begin{center}
\begin{tabular}{rS[table-format=3.2]S[table-format=3.2]S[table-format=3.2]S[table-format=3.2]S[table-format=3.2]}\toprule
& \multicolumn{5}{c}{Rate of Correct Recovery in \% for $m=1{,}000$} \\
\multirow{3}{*}{SNR} & {Alg. \ref{alg:dct_bound},} & {Alg. \ref{alg:dct_bound},} & {Alg. \ref{alg:dct_bound},} & {Alg. 2 in \cite{bit_plon}}	& {Alg. 2 in \cite{bit_plon},} \\
&  {$M=m$} & {$M=3m$} & {$M=3m$,} & &  \\ 
& & & {$m'\leq 3m$} & & {$m'\leq 3m$} \\ \midrule
0	& 51.6	& 88.0	& 0.0	& 83.1	& 68.0 \\
10	& 51.6	& 93.4	& 53.7	& 96.4	& 95.0 \\
20	& 99.4	& 100.0	& 84.5	& 100.0	& 99.7 \\
30	& 100.0	& 100.0	& 89.3	& 100.0	& 99.6 \\
40	& 100.0	& 100.0	& 94.8	& 100.0	& 99.8 \\
50	& 100.0	& 100.0	& 98.1	& 100.0	& 99.8 \\ \bottomrule
\end{tabular}
\caption{Rate of correct recovery of the support of $\x$ in \% for Algorithm \ref{alg:dct_bound} for $M=m$ and $M=3m$ and Algorithm 2 in \cite{bit_plon}, without bounding $m'$ and with $m'\leq 3m$, for 1{,}000 random input vectors with support length $m=1{,}000$ from Figure \ref{fig:error_dct_1000}.}
\label{tab:correct_recov_1000}
\end{center}
\end{table}
As Algorithm \ref{alg:dct_bound} and Algorithm 2 in \cite{bit_plon} tend to overestimate the support for noisy data, we consider $\x$ to be correctly recovered by $\x'$ in the second, third and fifth column if the support of $\x$ is contained in the support found by the sparse DCT algorithms. In the fourth and sixth column we additionally require that the support length $m'$ obtained by the procedures satisfies $m'\leq 3m$. Note that if $m$ is known exactly, Algorithm \ref{alg:dct_bound} will not overestimate the support length $m$.

For SNR values of 20 and greater all sparse DCT algorithms have very high rates of correct recovery. Algorithm \ref{alg:dct_bound} for bounded short support overestimates the support length by more than a factor three in less than 4\% of the cases for SNR values of 20 or more for $m=100$ and in less than 6 \% of the cases for SNR values of 40 or more for $m=1{,}000$. Algorithm 2 in \cite{bit_plon} never overestimates the support length for $m=100$ and in less than 1\% of the cases for $m=1{,}000$, both for SNR values of 20 or more.

\section*{Acknowledgement}
The authors gratefully acknowledge partial support for this work by the DFG in the framework of the GRK 2088.
\appendix

{\small
\bibliographystyle{abbrv}
\bibliography{library.bib}

\begin{thebibliography}{10}

\bibitem{akavia2014}
A.~Akavia.
\newblock {Deterministic sparse Fourier approximation via approximating
  arithmetic progressions}.
\newblock {\em IEEE Trans. Inform. Theory}, 60(3):1733--1741, 2014.

\bibitem{bittens2017}
S.~Bittens.
\newblock {Sparse FFT for Functions with Short Frequency Support}.
\newblock {\em Dolomites Res. Notes Approx.}, 10:43--55, 2017.

\bibitem{bit_plon}
S.~Bittens and G.~Plonka.
\newblock {Sparse Fast DCT for Vectors with One-block Support}.
\newblock \url{http://arxiv.org/abs/1803.05207}, 2018.

\bibitem{nam_code_dct}
S.~Bittens and G.~Plonka.
\newblock {Sparse Fast DCT for Vectors with One-block Support}.
\newblock
  \url{http://na.math.uni-goettingen.de/index.php?section=gruppe\&subsection=software},
  2018.

\bibitem{nam_code_dct_short}
S.~Bittens and G.~Plonka.
\newblock {Sparse Fast DCT for Vectors with Short Support}.
\newblock
  \url{http://na.math.uni-goettingen.de/index.php?section=gruppe\&subsection=software},
  2018.

\bibitem{bit_zha_iw}
S.~Bittens, R.~Zhang, and M.~A. Iwen.
\newblock {A Deterministic Sparse FFT for Functions with Structured Fourier
  Sparsity}.
\newblock \url{http://arxiv.org/abs/1705.05256}, 2017.

\bibitem{christlieb2016multiscale}
A.~Christlieb, D.~Lawlor, and Y.~Wang.
\newblock {A multiscale sub-linear time Fourier algorithm for noisy data}.
\newblock {\em Appl. Comput. Harmon. Anal.}, 40(3):553--574, 2016.

\bibitem{iwen}
M.~A. Iwen.
\newblock {Combinatorial Sublinear-Time Fourier Algorithms}.
\newblock {\em Found. Comput. Math.}, 10(3):303--338, 2010.

\bibitem{iwen_improved}
M.~A. Iwen.
\newblock {Improved Approximation Guarantees for Sublinear-Time Fourier
  Algorithms}.
\newblock {\em Appl. Comput. Harmon. Anal.}, 34(1):57--82, 1 2013.

\bibitem{PPST}
G.~Plonka, D.~Potts, G.~Steidl, and M.~Tasche.
\newblock {\em Numerical Fourier Analysis: Theory and Applications}.
\newblock Book manuscript, 2018.

\bibitem{plonka_dct}
G.~Plonka and M.~Tasche.
\newblock Fast and numerically stable algorithms for discrete cosine
  transforms.
\newblock {\em Linear Algebra Appl.}, 394:309 -- 345, 2005.

\bibitem{plonka_smallsupp}
G.~Plonka and K.~Wannenwetsch.
\newblock {A deterministic sparse FFT algorithm for vectors with small
  support}.
\newblock {\em Numer. Algorithms}, 71(4):889--905, 2016.

\bibitem{plonka_nonneg}
G.~Plonka and K.~Wannenwetsch.
\newblock {A sparse fast Fourier algorithm for real non-negative vectors}.
\newblock {\em J. Comput. Appl. Math.}, 321:532 -- 539, 2017.

\bibitem{plonka_sparse}
G.~Plonka, K.~Wannenwetsch, A.~Cuyt, and W.-s. Lee.
\newblock {Deterministic sparse FFT for M-sparse vectors}.
\newblock {\em Numer. Algorithms}, 78(1):133--159, 2018.

\bibitem{iwen2013improved}
B.~Segal and M.~Iwen.
\newblock {Improved sparse Fourier approximation results: faster
  implementations and stronger guarantees}.
\newblock {\em Numer. Algorithms}, 63(2):239--263, 2013.

\bibitem{matlab_idct}
{The MathWorks}.
\newblock {\textsc{Matlab}'s documentation of \texttt{idct}}.
\newblock \url{https://www.mathworks.com/help/signal/ref/idct.html}, 2017.

\bibitem{wang}
Z.~Wang.
\newblock {Fast Algorithms for the Discrete W Transform and for the Discrete
  Fourier Transform}.
\newblock {\em IEEE Trans. Acoust. Speech Signal Process.}, 32(4):803--816,
  1984.

\end{thebibliography}
}


\end{document}